\documentclass[11pt]{amsart}
\usepackage{amssymb,graphicx}
\usepackage{verbatim}
\usepackage{float}
\usepackage[dvipsnames]{xcolor}
\usepackage{amsmath}
\usepackage{cleveref}
\usepackage{autonum}
\usepackage{enumitem}
\newtheorem{theorem}{Theorem}[section]
\newtheorem{lemma}[theorem]{Lemma}
\newtheorem{proposition}[theorem]{Proposition}
\newtheorem{corollary}[theorem]{Corollary}

\theoremstyle{definition}
\newtheorem{definition}[theorem]{Definition}
\newtheorem*{nota}{Notation}
\newtheorem{remark}[theorem]{Remark}
\newtheorem{conjecture}[theorem]{Conjecture}
\newtheorem{example}[theorem]{Example}

\newcommand{\R}{{\mathbb R}}

\usepackage{refcount}
\usepackage{appendix}
\newcounter{cte}

\newif\ifrevisioncolors
\revisioncolorstrue

\newcommand{\versiontag}[1]{%
\par\smallskip
\noindent\fbox{\scriptsize\bfseries\sffamily #1}%
\par\nobreak\smallskip
}

\makeatletter
\newenvironment{oldversion}
{%
  \versiontag{OLD}%
  \begingroup
  \ifrevisioncolors\color{RoyalBlue}\fi
  \@afterindentfalse\@afterheading
}
{\endgroup}

\newenvironment{newversion}
{%
  \versiontag{NEW}%
  \begingroup
  \ifrevisioncolors\color{BrickRed}\fi
  \@afterindentfalse\@afterheading
}
{\endgroup}
\makeatother


\numberwithin{equation}{section}

\author[Bousquet]{Pierre Bousquet}
\address{Institut de Math\'ematiques de Toulouse, CNRS UMR 5219, Universit\'e de Toulouse, F-31062 Toulouse Cedex 9, France.}
\email{pierre.bousquet@math.univ-toulouse.fr}

\author[Mariconda]{Carlo Mariconda}
\address{Dipartimento di Matematica "Tullio Levi-Civita", Universit\`a degli Studi di Padova, Via Trieste 63, 35121 Padova, Italy.}
\email{carlo.mariconda@unipd.it}

\author[Treu]{Giulia Treu}
\address{Dipartimento di Matematica "Tullio Levi-Civita", Universit\`a degli Studi di Padova, Via Trieste 63, 35121 Padova, Italy.}
\email{giulia.treu@unipd.it}

\title{Beyond Constant Truncations: The Bounded Sobolev Approximation Gap}

\date{August 1, 2026}
\keywords{Sobolev approximation, bounded approximation, integral functionals,
Lavrentiev phenomenon, truncation methods}
\subjclass[2020]{49J45, 49N60}

\begin{document}
\begin{abstract}
We study the approximation of finite-energy Sobolev functions by bounded Sobolev functions for integral functionals associated with possibly signed Carath\'eodory integrands. The approximating sequence is required to converge strongly in the Sobolev space, while the corresponding energy densities converge in \(L^1\). Ordinary truncations may fail because they replace the gradient by zero on the tails, where the integrand need not be integrable. 
Our main contribution is a general method based on spatially dependent barriers  whose gradients can be adapted to the structure of the Lagrangian. 
\end{abstract}
\maketitle
\tableofcontents
\section{Introduction}
\subsection{Statement of the problem}
Let \(\Omega\) be a bounded open subset of \(\R^N\), let \(1\leq p\leq N\), and let
\[
\varphi\in W^{1,p}(\R^N)\cap L^\infty(\R^N).
\]
We set
\[
W^{1,p}_\varphi(\Omega):=\varphi+W^{1,p}_0(\Omega).
\]
Every \(u\in W^{1,p}_\varphi(\Omega)\) can be approximated by a sequence of bounded maps \((u_k)_{k\geq 1}\subset W^{1,p}_{\varphi}(\Omega)\cap L^{\infty}(\Omega)\), in the \(W^{1,p}\) sense:
\[
\lim_{k\to +\infty}\|u_k-u\|_{W^{1,p}(\Omega)} =0.
\]
Indeed, one only needs to define for every \(k\geq 1\),
\[
u_k:=T_k\circ u, \qquad \textrm{ where } T_k:t\in \R\mapsto \min(k, \max(-k,t)).
\]
Then, for every \(k>\|\varphi\|_{L^{\infty}(\Omega)}\), one has  \(u_k\in W^{1,p}_\varphi(\Omega)\cap L^{\infty}(\Omega)\) and, by the dominated convergence theorem,
\[
\|u_k-u\|_{W^{1,p}(\Omega)}^p=\||u|-k\|^p_{L^{p}([|u|\geq k])} + \|\nabla u\|^{p}_{L^{p}([|u|\geq k])}\longrightarrow 0, \quad \textrm{ when } k\to +\infty. 
\]
However, if one replaces the \(W^{1,p}\) norm by a more general (nonlinear) functional, this specific  sequence \((u_k)_{k\geq 1}\) obtained by truncating \(u\) with constants will not necessarily approximate the original map \(u\) in this new sense. More specifically, 
given a Carath\'eodory function
\[
L:\Omega\times\R\times\R^N\to\R,
\]
we consider the integral functional
\[
\mathcal L(u):=\int_\Omega L(x,u(x),\nabla u(x))\,dx
\]
on the admissible class
\[
\mathcal A:=\left\{u\in W^{1,p}_\varphi(\Omega):L(x,u,\nabla u)\in L^1(\Omega)\right\}.
\]
 We say that \(\mathcal L\) is \(L^\infty(\Omega)\)-regular at \(u\in\mathcal A\) if there exists a sequence
\[
(u_k)_{k\geq1}\subset\mathcal A\cap L^\infty(\Omega)
\]
that approximates \(u\) \emph{in norm}:
\[
u_k\longrightarrow u\quad\text{in }W^{1,p}(\Omega)
\]
and \emph{in energy}:
\[
L(x,u_k,\nabla u_k)\longrightarrow L(x,u,\nabla u)\quad\text{in }L^1(\Omega).
\]
We emphasize the fact that here, the convergence in energy means  convergence in \(L^1(\Omega)\) of the energy densities \(L(x,u_k,\nabla u_k)\), and not just  convergence in \(\R\) of the energies \(\mathcal{L}(u_k)\).

When \(p>N\), the Morrey embedding makes the problem trivial. The range \(p\leq N\) is substantially different.
First, in dimension \(N\geq2\), even if \(\mathcal A\neq\varnothing\), there need
not exist any bounded Sobolev function with finite energy:
The following example has been inspired by \cite[Example 3.3]{MaricondaTreu2020}.
\begin{example}
Assume that \(N\geq2\), \(1\leq p\leq N\), and that \(\Omega\)  is bounded, connected, and Lipschitz. Let \(u_0\in (W^{1,p}\setminus L^{\infty})(\Omega)\) 
and consider \(g(x,\xi)=|\nabla u_0(x)-\xi|^{N+1}\). Then for every \(b\in W^{1,p}(\Omega)\) such that
\[
\int_{\Omega}|\nabla b(x)-\nabla u_0(x)|^{N+1}\,dx <+\infty,
\]
the Poincar\'e-Wirtinger inequality implies that \(b-u_0\in W^{1,N+1}(\Omega)\), so that by the Morrey embedding, \(b-u_0\in L^{\infty}(\Omega)\). Then necessarily, \(b\not\in L^{\infty}(\Omega)\).
Hence \(g\) is a nonnegative Carath\'eodory function such that
\(g(\cdot,\nabla b)\notin L^1(\Omega)
\)
for every \(b\in W^{1,p}(\Omega)\cap L^\infty(\Omega)\).
\end{example}

Moreover, even if \(\mathcal{A}\cap L^{\infty}(\Omega)\not=\emptyset\), there may exist \(u\in \mathcal{A}\) such that \((T_k(u))_{k\geq 1}\) does not approximate \(u\) in energy. In Example~\ref{ex-contrex-ct-barrier}, we construct a Carathéodory function \(L:\Omega\times \R\times \R^N\to \R\)  and \(u\in \mathcal{A}\) such that \(\mathcal{L}\) is \(L^{\infty}(\Omega)\)-regular at \(u\) and yet,
\[
\lim_{k\to +\infty}\int_{\Omega}|L(x,T_k(u)(x),\nabla T_k(u)(x))-L(x,u(x),\nabla u(x))|\,dx =+\infty.
\]
This illustrates the fact that considering the sequence \((T_k(u))_{k\geq 1}\) is not always the good strategy to approximate \(u\) in norm and in energy. 

The main aim of this paper is to identify 
simple criteria on \(L\) and \(u\in \mathcal{A}\) ensuring that the corresponding functional \(\mathcal{L}\) is \(L^{\infty}(\Omega)\)-regular  at \(u\). We thus explicitly construct a variety of approximating sequences which differ from \((T_k(u))_{k\geq 1}\), depending on the specific properties of \(L\) and \(u\). 

\subsection{Motivation: the Lavrentiev phenomenon}
In the Calculus of variations, the Lavrentiev phenomenon refers to the fact that the infimum of a functional \(\mathcal{L}\) may depend on the set of admissible functions on which \(\mathcal{L}\) is minimized. More specifically, assuming for one instant that \(\varphi\in C^{1}(R^N)\), it may happen that for a nonnegative Carathéodory function \(L:\Omega\times \R\times \R^N\), one has 
\[
\inf_{C^{1}_{\varphi}(\overline{\Omega})}\mathcal{L}>\inf_{W^{1,p}_\varphi(\Omega)}\mathcal{L},
\]
where \(C^{1}_{\varphi}(\overline{\Omega})\) denotes the set of those functions \(u\in C^{1}(\overline{\Omega})\) agreeing with \(\varphi\) on \(\partial \Omega\). In many situations (see e.g. \cite{BorowskiChlebicka2022, Bousquet-Annali, BousquetMaricondaTreu2014, BousquetMaricondaTreu2024,   CorboEspositoDeArcangelis, Zhikov-2006}), the first step to discard such a phenomenon is to prove that
\[
\inf_{W^{1,p}_\varphi(\Omega)\cap L^{\infty}(\Omega)}\mathcal{L}=\inf_{W^{1,p}_\varphi(\Omega)}\mathcal{L}.
\]
To establish this identity, one often aims at proving that \(\mathcal{L}\) is \(L^{\infty}(\Omega)\)-regular at \emph{every} \(u\in \mathcal{A}\). When this property holds, we say that there is no \emph{Lavrentiev gap} between \(W^{1,p}(\Omega)\) and \(W^{1,p}(\Omega)\cap L^{\infty}(\Omega)\).

\subsection{Main results}
Throughout the paper, we use the notation
\[
\overline L(x,t,\xi):=|L(x,t,\xi)|+|\xi|^p.
\]

\subsubsection{General Lagrangians}

We first present sufficient conditions ensuring the \(L^{\infty}(\Omega)\)-regularity of \(\mathcal{L}\), when the Lagrangian \(L:\Omega\times \R\times \R^N\to \R\) has  a general dependence on the three variables \(x\in \Omega\), \(t\in \R\) and \(\xi\in \R^N\). In our first main result, those sufficient conditions are formulated  in terms of the summability of \(L(x,t,\nabla b)\) for some specific functions  \(b\in W^{1,p}(\Omega)\cap L^{\infty}(\Omega)\). 

\begin{theorem}
\label{thm-gen-Lag-1}
Let \(q\in L^1(\Omega)\) and \(h:\R\to[0,+\infty)\) a Borel map. Assume that at least one of the two following assumptions is satisfied:
\begin{itemize}
\item[(H1)] For every \(s\in\R\), there exist an integer \(m_s\geq1\), vectors
\(
\xi_1^s,\dots,\xi_{m_s}^s\in\R^N
\) containing \(0\) in their convex hull 
and  \(\delta_s>0\) such that, for a.e. \(x\in\Omega\) and every \(t\in[s-\delta_s,s+\delta_s]\),
\begin{equation}\label{eq-local-balanced-control}
\sum_{j=1}^{m_s}\overline L(x,t,\xi_j^s)
\leq h(s)+q(x).
\end{equation}
\item[(H2)] There exists \(b\in W^{1,p}(\Omega)\cap L^\infty(\Omega)\) such that for a.e. \(x\in\Omega\) and every \(s,t\in\R\) satisfying
\(
|s-t|\leq \operatorname*{ess\,osc}_{\Omega}b\),
one has
\[
\overline L(x,t,\nabla b(x))
\leq h(s)+q(x).
\]
\end{itemize}
Then \(\mathcal L\) is \(L^\infty(\Omega)\)-regular at every \(u\in\mathcal A\) such that \(h\circ u\in L^{1}(\Omega)\).
\end{theorem}
Stated otherwise,  assumption (H1) involves, for every value \(s\) of the state variable, a finite family of affine functions \(b_{j}^{s}(x)=\langle \xi_{j}^{s},x\rangle\)  such that \(0\in \operatorname{co}(\nabla b_{j}^s(x), \dots, \nabla b_{m_s}^{s}(x))\) for every \(s\in \R\) and \(x\in \Omega\). In contrast,  just one suitable function  \(b\in W^{1,p}(\Omega)\cap L^{\infty}(\Omega)\) is required in (H2). In both cases, a summability condition must be satisfied by the functions \(L(x,t,\nabla b_{j}^{s})\) (under (H1)) or \(L(x,t,\nabla b)\) (under (H2)). The main difference between the two cases is the dependence with respect to the state variable: roughly speaking,  we require in (H1) that \(L(x,\cdot, \nabla b_{j}^s)\) be   continuous in the state variable, uniformly with respect to \(x\), while in (H2),  \(L(x,\cdot,\nabla b)\) is assumed to be almost constant. 

In order to illustrate Theorem~\ref{thm-gen-Lag-1}, let us first consider the situation when (H1) is satisfied. 
It is sometimes possible to take the \emph{same} vectors \(\xi_{j}^s\) independently of \(s\), leading to a much simpler form of (H1), see~Corollary~\ref{prop-H(u,0)} below. For instance, consider the case when \(L(x,t,\xi)=t^4 g(\xi)\) and \(g\) is a smooth function compactly supported  in the unit ball \(B_1(0)\subset \R^N\) such that \(g(0)=1\). Then, \(L(x,t,\xi)=0\) for every \(\xi \not\in B_1(0)\) so that (H1) is satisfied in that case, since we only need to take \(m_s=2\) and  \(\xi_{1}^{s}=-\xi_{2}^{s}\) equal to any fixed vector in \(\R^N\setminus B_{1}(0)\). For a further discussion on that specific Lagrangian, where the standard truncation method does not apply, see Example~\ref{ex-contrex-ct-barrier}.

The  simplest situation where assumption (H2) is satisfied occurs when the function \(b\) satisfies \(L(x,t,\nabla b(x))=0\) for a.e. \(x\in \Omega\) and every \(t\in \R\). This is the case for instance when \(L(x,t,\xi)=\sum_{i=1}^{m}a_i(x,t)g_i(\xi)\) and there exists \(\xi_0\in \R^N\) such that \(g_i(\xi_0)=0\) for every \(i=1, \dots, m\). Another simple situation where (H2) holds, arises when for some \(\xi_*\in \R^N\), the function \(t\mapsto L(x,t, \xi_*)\) satisfies a doubling growth condition, see~Corollary~\ref{cor-single-direction} for a precise statement.

In the next result, we present another framework where any \(u\in \mathcal{A}\) can be approximated  by bounded Sobolev functions in norm and in energy. Here, the crucial assumption is related to the growth properties of \(L\) with respect to the state variable. 

\begin{theorem}\label{coro-other-truncations-i}
Assume that there exist \(\theta^+, \theta^- \in [0,1]\), two  sequences \((m_{i}^+)_{i\geq1}, (m_{i}^-)_{i\geq1}\) converging to
\(+\infty\),   a sequence  \((\varepsilon_i)_{i\geq1}\) of positive numbers, a summable function \(q\in L^{1}(\Omega)\)
and a constant \(C\geq 0\) such that, for every \(i\geq1\), for a.e.
\(x\in\Omega\), and for every \(\xi\in\mathbb R^N\),
\begin{equation}\label{eq3098i}\max_{s\in[m_{i}^+-\varepsilon_i,m_{i}^+]}\bigl(|L(x,s,\theta^+ \xi)|+ |L(x,s,-\theta^+\xi)|\bigr)
\leq
C\inf_{s\in[m_{i}^+,+\infty)}|L(x,s,\xi)|+q(x)
\end{equation}
and
\begin{equation}\label{eq3107i}
\max_{s\in[-m_{i}^-,-m_{i}^-+\varepsilon_i]}\bigl(|L(x,s,\theta^-\xi)|+
|L(x,s,-\theta^-\xi)| \bigr)
\leq
C\inf_{s\in(-\infty,-m_{i}^-]}|L(x,s,\xi)|+q(x).
\end{equation}
Then \(\mathcal L\) is \(L^\infty(\Omega)\)-regular at every
\(u\in\mathcal A\).
\end{theorem}

Assumptions \eqref{eq3098i} and \eqref{eq3107i} are satisfied (with \(\theta^\pm=1\)) when \(t\mapsto |L(x,t,\xi)|\) is nondecreasing on \(\R^+\) and nonincreasing on \(\R^-\), and  \(L\) is even with respect to \(\xi\),   see~Corollary~\ref{cor-monotone-tails} for a precise (and more general) statement.

Both Theorem~\ref{thm-gen-Lag-1} and Theorem~\ref{coro-other-truncations-i} can be formulated in a simpler way when we further assume that the Lagragian has a specific structure. In this article, we focus our attention on two specific cases, namely autonomous and product-type Lagrangians.

\subsection{Autonomous Lagrangians}

We say that \(L\) is  \emph{autonomous} when there exists a continuous function \(H:\R\times \R^N \to \R\) such that \(L(x,t,\xi)=H(t,\xi)\) for every \((x,t,\xi)\in \Omega \times \R\times \R^N\).
In that case, by relying on constant truncations, one can easily prove (see Proposition~\ref{cor-lu0} below) that \(\mathcal{L}\) is \(L^{\infty}(\Omega)\)-regular at any \(u\in \mathcal{A}\) such that 
\begin{equation}\label{eq285}
H(u,0)\in L^{1}(\Omega).
\end{equation}
However, the  above assumption is not implied by the fact that \(u\in \mathcal{A}\), see Example~\ref{ex-contrex-ct-barrier} or Example~\ref{ex:^q}. It is thus convenient to identify simple conditions on \(H\) which ensure that any \(u\in \mathcal{A}\) satisfies~\eqref{eq285}. Such a condition is given in the first item of the following statement:

\begin{theorem}\label{thm-synt-autonomous-1}
Assume that
\[
L(x,t,\xi)=H(t,\xi)
\]
for some continuous nonnegative function \(H:\R\times\R^N\to [0,+\infty)\).
Assume  further that at least one of the two following assumptions is satisfied:
\begin{itemize}
\item[(HA)] For every \(t\in \R\), there exists \(\zeta\in \R^N\) such that
\[
\forall \xi \in \R^N, \quad H(t,\xi)\geq H(t,0)+\langle \zeta, \xi \rangle.
\]
\item[(HB)] There exists \(C\geq 1\) such that for every \(t\in \R\), for every \(\xi\in \R^N\), one has
\[
H(t,-\xi)\leq C(H(t,\xi)+1).
\]
\end{itemize}
Then, \(\mathcal{L}\) is \(L^{\infty}(\Omega)\)-regular at every \(u\in \mathcal{A}\).
\end{theorem}

The first condition means that for every \(t\in \R\), the function \(\xi \mapsto H(t,\xi)\) has a convex subgradient at the origin. This is in particular the case when \(H(t,\cdot)\) is convex (or at least, coincides with its convex envelope at \(0\)) and we thus recover~\cite[Proposition~5.3]{Bousquet-Annali}. It turns out that (HA) is just one particular case of a more general condition where one assumes that \(\xi\mapsto H(t,\xi)\) has  convex subgradients at finitely many points \(\xi_1, \dots, \xi_m\) containing \(0\) in their convex hull, together with a suitable summability condition which is automatically satisfied when \(\xi_1=\dots=\xi_m=0\), see~Theorem~\ref{cor-barycentric-subgradients} for a precise statement. 
 
Condition (HB) is trivially satisfied when \(H\) is even with respect to \(\xi\), in particular when \(H\) depends on \(\xi\) only through a norm of that variable. 
 
Theorem~\ref{thm-synt-autonomous-1} can be easily extended to the case when \(H\) is not assumed to be nonnegative anymore, see Theorem~\ref{th-real-valued-affine-minorant} for (HA) and Corollary~\ref{corollary-autonomous-symetric} for (HB).

\subsection{Product type Lagrangians}
Another class of Lagrangians for which Theorem~\ref{thm-gen-Lag-1} and Theorem~\ref{coro-other-truncations-i} yield interesting consequences is the set of those functions \(L:\Omega\times \R\times \R^N\to \R\) having the following product-type structure:  there exist a continuous function \(a:\R\to \R\) and a Carath\'eodory function \(g:\Omega\times \R^N\to \R\) such that for every \((x,t,\xi)\in \Omega\times \R\times \R^N\),
\[
L(x,t,\xi)=a(t)g(x,\xi).
\]
Given \(u\in \mathcal{A}\) such that 
\begin{equation}\label{eq324}
(a\circ u)g(x,0)\in L^{1}(\Omega),
\end{equation} 
one can easily approximate \(u\) in norm and in energy by bounded Sobolev functions, provided that \(g(x,0)\in L^{1}(\Omega)\), see Proposition~\ref{lm-Giotto} below, the proof of which relies on the standard truncation approach. However, assumption~\eqref{eq324} is not automatically implied by the condition \(u\in \mathcal{A}\). In the next result, we provide general assumptions on \(a\) and \(g\) which entail that  \(\mathcal{L}\) is \(L^{\infty}(\Omega)\)-regular at \emph{every} \(u\in \mathcal{A}\).

\begin{theorem}\label{thm-synt-product-1}
Assume that \(L\) has a product-type structure as above and that at least  one of the following assumptions hold:
\begin{itemize}
\item[(Hi)] \(g(\cdot,0)\in L^\infty(\Omega)\) and there exists \(\alpha>0\) such that
\[
|g(x,\xi)|\geq\alpha
\qquad\text{for a.e. }x\in\Omega\text{ and every }\xi\in\R^N;
\]
\item[(Hii)] there exists \(\xi_0\in\R^N\) such that
\(g(x,\xi_0)=0\) for a.e. \(x\in\Omega\);
\item[(Hiii)] \(a\equiv 1\) and there exists \(b\in W^{1,p}(\Omega)\cap L^\infty(\Omega)\) such that
\(g(\cdot,\nabla b)\in L^1(\Omega)\);
\item[(Hiv)] \(g(\cdot,0)\in L^1(\Omega)\) and  there exists \(c>0\) such that for a.e. \(x\in\Omega\) and every \(\xi\in\mathbb R^N\),
\[
|g(x,-\xi)|\leq c|g(x,\xi)|
\]
\end{itemize}
Then \(\mathcal L\) is \(L^\infty(\Omega)\)-regular at every \(u\in\mathcal A\).
\end{theorem}

\subsection{Strategies of the proofs}

The main idea of the paper is to replace constant truncation levels by spatially dependent barriers that may depend on the function \(u\in \mathcal{A}\) to be approximated. More specifically, given \(L:\Omega\times \R\times \R^N\to \R\) and \(u\in \mathcal{A}\), we construct two families of functions \((c_k^+)_{k\geq 1}\) and \((c_k^-)_{k\geq 1}\) in \(W^{1,p}(\Omega)\cap L^{\infty}(\Omega)\) that uniformly converge respectively to \(+\infty\) and  \(-\infty\). We then say that \(c_{k}^+\) (resp. \(c_{k}^-\)) are upper  (resp. lower) barriers. For instance, in the standard truncation approach by constant functions, one simply takes \(c_{k}^+\equiv k\) and \(c_{k}^-\equiv -k\). For  general barriers \(c_{k}^{\pm}\), we  consider  the approximating sequence:
\begin{equation}\label{eq528}
u_k:=\min\{c_k^+,\max\{u,c_k^-\}\}, \qquad k\geq 1.
\end{equation}
We call the set \([u\geq c_{k}^+]\) (resp. \([u\leq c_{k}^-]\)) the positive tail (resp. the negative tail). The tails thus correspond to the truncated regions, where the gradients of \(u_k\) are \(\nabla c_k^+\) and \(\nabla c_k^-\), rather than zero as in the case when one truncates with constant functions. These gradients \(\nabla c_{k}^{\pm}\) can therefore be chosen according to the structure of the Lagrangian.

Lemma~\ref{lm-classic} gives the basic barrier criterion implying that the functions \(u_k\) defined in \eqref{eq528} converge to \(u\) in norm and in energy:
\[
\lim_{k\to +\infty}\left(\int_{[u>c_k^+]}
\left(|L(x,c_k^+,\nabla c_k^+)|+|\nabla c_{k}^+|^p\right)\,dx + \int_{[u<c_k^-]}
\left(|L(x,c_k^-,\nabla c_k^-)|+|\nabla c_{k}^-|^p\right)\,dx\right)=0.
\]

The heart of the matter is then to construct barriers \(c_{k}^{\pm}\) that satisfy this criterion.
At a technical level, the three main outcomes of this paper can be described as follows:
\begin{enumerate}
\item Construction of barriers with small oscillations, taking their gradients in a finite set;
\item A summability criterion for functions of the form \(x\mapsto H(u(x), \xi)\) with \(\xi\in \R^N\);
\item Construction of  barriers of the form \(\pi\circ u\), where  \(\pi\) is a sawtooth function with small oscillations.
\end{enumerate}
As for the first technique, we rely on  the following basic problem in convex integration theory (see~e.g. \cite{DacorognaMarcellini}): given a finite set \(\xi_1, \dots, \xi_m \in \R^N\), is it possible to construct a function \(b\in W^{1,\infty}_0(\Omega)\) such that \(\nabla b(x)\in \{\xi_1,\dots, \xi_m\}\) for a.e. \(x\in \Omega\) ? The difficulty is that \(b\) must vanish on the boundary of \(\Omega\). 
One can positively answer  this question when \(0\) belongs to the  interior of \(\operatorname{co} (\xi_1, \dots, \xi_m)\). By translating and scaling \(b\), we can thus produce functions that converge  uniformly to \(\pm\infty\), taking their gradients in \(\xi_1, \dots, \xi_m\),  and with arbitrarily small oscillations. We can apply this observation to construct barriers provided that there exist vectors \(\xi_1, \dots, \xi_m\)  containing \(0\) in their convex hull and such that \(L(x,u,\xi_j)\) is  summable. This technique, that we exploit in Lemma~\ref{lm-oscillating-functions}, is the main ingredient in the proof of Theorem~\ref{thm-gen-Lag-1} under assumption (H1) and of Theorem~\ref{thm-synt-autonomous-1} under assumption (HB).

The second technique, which is encapsulated in Lemma~\ref{prop-affine-minorant}, is not new, since we have already used it in our previous papers \cite{Bousquet-Annali} and \cite{MaricondaTreu2020}. However, we generalize it and substantially extend its range of applications. More specifically, Lemma~\ref{prop-affine-minorant} states the following summability property: if the continuous  maps \(H:\R\times \R^N\to [0,+\infty)\) and \(K:\R\to [0,+\infty)\) satisfy a subgradient inequality of the form:
\begin{equation}\label{eq554}
H(t,\xi) \geq K(t)+\langle \zeta_t,\xi\rangle, \qquad \forall \xi\in \R^N,
\end{equation}
and if \(H(u,\nabla u)\in L^{1}(\Omega)\), then \(K(u)\in L^{1}(\Omega)\).  In \eqref{eq554}, the vectors \(\zeta_t\) may depend on \(t\) but not on \(\xi\). 
In a nutshell, this lemma says that the term \(\langle \zeta_t,\xi\rangle\) does not interfere in the summability of  \(K(u)\). We emphasize that the boundedness of the boundary condition \(\varphi\) plays a crucial role here, see~\cite[Example 4.8]{BousquetMaricondaTreu2014}. 
We apply Lemma~\ref{prop-affine-minorant} not just when \(H(t,\cdot)\) has a convex subgradient at the origin (Theorem~\ref{thm-synt-autonomous-1} under (HA)) but also when \(H(t,\cdot)\) has convex subgradients at finitely many points containing \(0\) in their convex hulls (see~Theorem~\ref{cor-barycentric-subgradients}). This allows some freedom to construct barriers when the Lagrangian is not convex near the origin. 

Finally, the third technique, based on sawtooth functions, is the key ingredient in the proof of Theorem~\ref{coro-other-truncations-i}. Let \(u\in \mathcal{A}\) such that \(L(x,u, -\nabla u)\) is also summable. Let \(b=\pi\circ u\), where \(\pi:\R\to \R\) is a function such that \(\pi'(t)\in \{\pm 1\}\) almost everywhere on \(\R\). Then, \(\nabla b(x)=\pm \nabla u(x)\) almost everywhere on \(\Omega\). 
We can thus deduce that \(L(x,u, \nabla b)\) is summable. Actually, one takes \(\pi(t)=t\) on a large interval \([-m,m]\) and  \(\pi(t)\) close to constants when \(t\) is large and when \(-t\) is large. On the sets \(]-\infty,-m]\) and \([m,+\infty[\), the function \(\pi\) is thus a sawtooth function. By exploiting the continuity of \(L\) with respect to the state variable, it follows that \(L(x,b,\nabla b)\) is summable on the tails of \(u\). One can thus construct the desired upper and lower barriers.

\subsection{Plan of the paper}
 Section~2.1 develops the general barrier  methods.  
 It contains Lemma~\ref{lm-classic} and one of its useful variants Lemma~\ref{lm-translated-oscillators}. We also present there Proposition~\ref{cor-lu0} and Proposition~\ref{lm-Giotto}, that both involve truncations by constant functions.

 In Section~2.2, we construct barriers with small oscillations and prescribed gradients. This leads to the proof of Theorem~\ref{thm-gen-Lag-1} under  (H1) (see Theorem~\ref{prop-della-francesca}). As a consequence, we also obtain Theorem~\ref{thm-synt-autonomous-1} under (HB) (see Corollary~\ref{corollary-autonomous-symetric}).

 In Section~2.3, we construct barriers with large oscillations for Lagrangians which are almost constant with respect to the state variable.  This implies Theorem~\ref{thm-gen-Lag-1} under  (H2) (see Proposition~\ref{prop-single-oscillator}). We then derive Theorem~\ref{thm-synt-product-1} under (Hi) and (Hii) (see Corollary~\ref{propositionH1H2}), and also Corollary~\ref{cor-single-direction}. 

Section~3 is devoted to autonomous Lagrangians and to the affine-minorant method.  It contains the summability result of Lemma~\ref{prop-affine-minorant} and its consequences Theorem~\ref{th-real-valued-affine-minorant} and Theorem~\ref{cor-barycentric-subgradients}. The former implies, for nonnegative autonomous Lagrangians, Theorem~\ref{thm-synt-autonomous-1} under (HA).

Section~4 deals with some specific classes of non-autonomous Lagrangians.  First, Theorem~\ref{prop:nonaton2026} settles the case of the Lagrangians not depending on the state variable \(s\). As a consequence, we obtain  Theorem~\ref{thm-synt-product-1} under (Hiii). We also state there Theorem~\ref{coro-other-truncations} which is useful for Lagrangians that are comparable, in a suitable sense, to autonomous Lagrangians. The main part of this section is devoted to the construction of barriers based on  left-compositions, namely functions of the form \(\pi_i\circ u\), where \(\pi_i(t)\) is, for  large or small values of \(t\in \R\), a sawtooth function with  small oscillations.  This leads to the proof of Theorem~\ref{coro-other-truncations-i} (see Theorem~\ref{coro-other-truncations}) and Theorem~\ref{thm-synt-product-1} under (Hiv) (see Corollary~\ref{cor-dellafrancesca}). 

\section*{Notation} 

Let \(\Omega\) be a bounded  open set in \(\R^N\) and \(p\in [1,N]\). 
We denote by \(\mathcal{C}^{1}_c(\Omega)\)  the space of \(\mathcal{C}^{1}\) functions with compact support on \(\Omega\). We then introduce the space \(W^{1,p}_{0}(\Omega)\) defined as the closure of \(\mathcal{C}^{1}_c(\Omega)\) in \(W^{1,p}(\Omega)\). Given \(\varphi\in W^{1,p}(\R^N)\cap L^{\infty}(\R^N)\), we denote by 
 \(W^{1,p}_{\varphi}(\Omega)\) the set
\[
 \varphi+W^{1,p}_{0}(\Omega)=\{u+\varphi : u\in W^{1,p}_0(\Omega)\}.
\]
Then every \(u\in W^{1,p}_\varphi(\Omega)\) belongs to \(L^{p^\#}(\Omega)\), where we have set
\[
p^\#:=
\begin{cases}
\dfrac{Np}{N-p},&\text{if }1\leq p<N,\\
N,&\text{if }p=N.
\end{cases}
\]

Let \(L:\Omega\times \R\times \R^N \to \R\) be a Carath\'eodory function and
\[
\mathcal{L}(u)=\int_{\Omega}L(x,u(x),\nabla u(x))\,dx, \qquad u\in W^{1,p}(\Omega).
\]

\begin{nota}
For every \((x,t,\xi)\in\Omega\times\R\times\R^N\), we set
\[
\overline L(x,t,\xi):=|L(x,t,\xi)|+|\xi|^p.
\]
This notation will be used throughout the paper.
\end{nota}

We denote by \(\mathcal{A}\) the set of those \(u\in W^{1,p}_{\varphi}(\Omega)\) such that \(L(x, u,\nabla u)\) is summable on \(\Omega\). Then \(\mathcal{L}\) is well-defined on \(\mathcal{A}\).

\begin{definition}
We say that \(\mathcal{L}\) is \(L^{\infty}(\Omega)\)-regular at some \(u\in \mathcal{A}\) when there exists a sequence \((u_k)_{k}\subset \mathcal{A}\cap L^{\infty}(\Omega)\) converging to \(u\) in \(W^{1,p}(\Omega)\) and such that \((L(x,u_k, \nabla u_k))_{k\geq 1}\) converges to \(L(x,u, \nabla u)\) in \(L^{1}(\Omega)\).
\end{definition}

\begin{nota}
For every real-valued function \(g\), we denote its positive and negative parts by
\[
g_+(t):=\max\{g(t),0\}, \qquad g_-(t):=\max\{-g(t),0\}.
\]    
For \(c\in L^\infty(\Omega)\), we set
\[
\operatorname*{ess\,osc}_{\Omega}c
:=\operatorname*{ess\,sup}_{\Omega}c-\operatorname*{ess\,inf}_{\Omega}c.
\]
\end{nota}

\section{General regularity criteria}
\subsection{Barrier truncations}

The following lemma provides the basic truncation mechanism used throughout the paper: suitable upper and lower barriers with evanescent energy on the corresponding tails yield bounded  approximations.
\begin{lemma}\label{lm-classic}
Let \(u\in\mathcal A\). Assume that there exist two sequences
\[
(c_k^+)_{k\geq1},\ (c_k^-)_{k\geq1}
\subset W^{1,p}(\Omega)\cap L^\infty(\Omega)
\]
such that
\[
\operatorname*{ess\,inf}_{\Omega}c_k^+\longrightarrow+\infty,
\qquad
\operatorname*{ess\,sup}_{\Omega}c_k^-\longrightarrow-\infty,
\]
and
\begin{align}
\int_{[u>c_k^+]}
\overline L(x,c_k^+,\nabla c_k^+)\,dx
&\longrightarrow0,
\label{eq-Padova-1}\\
\int_{[u<c_k^-]}
\overline L(x,c_k^-,\nabla c_k^-)\,dx
&\longrightarrow0.
\label{eq-Padova-2}
\end{align}
Then \(\mathcal L\) is \(L^\infty(\Omega)\)-regular at \(u\). More specifically, the functions \(u_k=\min(c_{k}^+, \max(u, c_{k}^-))\) converge in norm and in energy to \(u\).
\end{lemma}
\begin{proof}
Set
\[
M_k^+:=\operatorname*{ess\,inf}_{\Omega}c_k^+,
\qquad
M_k^-:=-\operatorname*{ess\,sup}_{\Omega}c_k^-.
\]
Then
\[
\lim_{k\to +\infty}M_k^\pm = +\infty. 
\]
Since \(\varphi\in L^\infty(\Omega)\), after discarding finitely many
terms, we have
\[
M_k^+>\|\varphi\|_{L^\infty(\Omega)}
\qquad\text{and}\qquad
M_k^->\|\varphi\|_{L^\infty(\Omega)}.
\]
Since
\(
c_k^+\geq M_k^+\)
 and 
\(c_k^-\leq -M_k^-\) a.e. in \(\Omega\),
it follows that
\[
c_k^-\leq\varphi\leq c_k^+
\qquad\text{a.e. in }\Omega.
\]
For every such \(k\), define
\[
u_k(x):=\min\{c_k^+(x),\max\{u(x),c_k^-(x)\}\}
=\begin{cases}
c_k^-(x),&\text{if }u(x)<c_k^-(x),\\
u(x),&\text{if }c_k^-(x)\leq u(x)\leq c_k^+(x),\\
c_k^+(x),&\text{if }u(x)>c_k^+(x).
\end{cases}
\]
Since \(u,c_k^\pm\in W^{1,p}(\Omega)\), the stability of \(W^{1,p}(\Omega)\) under pointwise minima and maxima gives
\[
u_k\in W^{1,p}(\Omega)\cap L^\infty(\Omega).
\]
Moreover, the inequalities \(c_k^-\leq\varphi\leq c_k^+\) imply that
\[
u_k\in W^{1,p}_\varphi(\Omega).
\]
Since
\(
[u>c_k^+]\subset[u>M_k^+]\),
\([u<c_k^-]\subset[u<-M_k^-]\), 
we have
\[
|u_k-u|^p
\leq |u|^p
\left(
\chi_{[u>M_k^+]}+\chi_{[u<-M_k^-]}
\right).
\]
Since \(u\in L^p(\Omega)\) and \(M_k^\pm\to+\infty\), it follows that
\[
u_k\longrightarrow u
\quad\text{in }L^p(\Omega).
\]
Furthermore,
\[
\begin{aligned}
\int_\Omega\bigl(
&|L(x,u,\nabla u)-L(x,u_k,\nabla u_k)|
+|\nabla u-\nabla u_k|^p
\bigr)\,dx\\
\leq{}&
\int_{[u>c_k^+]}
\Bigl(
|L(x,u,\nabla u)|
+|L(x,c_k^+,\nabla c_k^+)|
+2^{p-1}\bigl(|\nabla u|^p+|\nabla c_k^+|^p\bigr)
\Bigr)\,dx\\
&+\int_{[u<c_k^-]}
\Bigl(
|L(x,u,\nabla u)|
+|L(x,c_k^-,\nabla c_k^-)|
+2^{p-1}\bigl(|\nabla u|^p+|\nabla c_k^-|^p\bigr)
\Bigr)\,dx\\
&\le 2^{p-1}\left(\int_{[u>c_k^+]}\overline L(x,c_k^+,\nabla c_k^+)\,dx+ \int_{[u<c_k^-]}\overline L(x,c_k^-,\nabla c_k^-)\,dx+\right.\\&\left.\qquad\qquad\qquad\int_{[u>c_k^+]\cup [u<c_k^-]}\overline L(x,u,\nabla u)\,dx\right).
\end{aligned}
\]
Since
\(
[u>c_k^+]\subset[u>M_k^+]\),
\([u<c_k^-]\subset[u<-M_k^-]\)
and
\(
\overline L(x,u,\nabla u)\in L^1(\Omega)\),
we obtain
\[
\lim_{k\to +\infty}\int_{[u>c_k^+]\cup [u<c_k^-]}
\overline L(x,u,\nabla u)\,dx
=0.
\]
Combining these estimates with \eqref{eq-Padova-1}--\eqref{eq-Padova-2}, we obtain
\[
\nabla u_k\longrightarrow\nabla u
\quad\text{in }L^p(\Omega;\mathbb R^N)
\]
and
\[
L(x,u_k,\nabla u_k)
\longrightarrow
L(x,u,\nabla u)
\quad\text{in }L^1(\Omega).
\]
Thus \(u_k\to u\) in \(W^{1,p}(\Omega)\), and the conclusion follows.
\end{proof}
\begin{remark}
Lemma~\ref{lm-classic} is closely related to
\cite[Theorem~3.5]{MaricondaTreu2020}, where the upper and lower
barriers are affine functions with fixed gradients. In the present
lemma, the barriers may be arbitrary functions in
\(W^{1,p}(\Omega)\cap L^\infty(\Omega)\). No uniform bound on their
oscillations is required. 
\end{remark}

\begin{remark}\label{rk-b=0}
The most standard application of Lemma~\ref{lm-classic} occurs when it is possible to choose two sequences \(M_k^+,M_k^-\to+\infty\) such that
\[
\lim_{k\to +\infty}\int_{[u>M_k^+]}\overline L(x,M_k^+,0)\,dx =0,
\qquad
\lim_{k\to +\infty}\int_{[u<-M_k^-]}\overline L(x,-M_k^-,0)\,dx=0,
\]
and taking
\[
c_k^+\equiv M_k^+,
\qquad
c_k^-\equiv-M_k^-.
\]
In particular, such sequences exist when
\[
\liminf_{t\to+\infty}
\int_{[u>t]}\overline L(x,t,0)\,dx=0,
\qquad
\liminf_{t\to+\infty}
\int_{[u<-t]}\overline L(x,-t,0)\,dx=0.
\]
\end{remark}

Another situation where the  choice of barriers given in Remark~\ref{rk-b=0} is possible is detailed in the following statement dealing with \emph{autonomous} Lagrangians. Remember that  \(L\) is said to be autonomous when there exists a continuous function \(H:\R\times \R^N \to \R\) such that \(L(x,t,\xi)=H(t,\xi)\) for every \((x,t,\xi)\in \Omega \times \R\times \R^N\).
In the autonomous setting, we use the corresponding notation
\[
\overline H(t,\xi):=|H(t,\xi)|+|\xi|^p.
\]
Thus
\[
\overline L(x,t,\xi)=\overline H(t,\xi)
\]
for every \((x,t,\xi)\in\Omega\times\R\times\R^N\).

\begin{proposition}\label{cor-lu0}
Assume that
\[
L(x,t,\xi)=H(t,\xi)
\]
for some continuous function \(H:\R\times\R^N\to\R\). Then, \(\mathcal{L}\) is \(L^{\infty}(\Omega)\)-regular at every \(u\in \mathcal{A}\) such that \(H(u, 0)\in L^{1}(\Omega)\).
\end{proposition}
\begin{proof}
Let \(u\in \mathcal{A}\) such that \(H(u, 0)\in L^{1}(\Omega)\).
Apply Lemma~\ref{lm-mkpm} to the measurable function \(h:t\mapsto H(t,0)\) to get two sequences \((M_{k}^+)_{k\geq 1}\) and \((M_{k}^-)_{k\geq 1}\) converging to \(+\infty\) and such that 
\[
\lim_{k\to +\infty}|[u>M_k^+]|\,|H(M_k^+,0)|=0,\quad 
\lim_{k\to +\infty}|[u<-M_k^-]|\,|H(-M_k^-,0)|=0.
\]
Then, in view of Remark~\ref{rk-b=0} and  Lemma~\ref{lm-classic},  the functions \(u_k:=\min(M_k^+, \max (u, -M_k^-))\) belong to \(W^{1,p}_{\varphi}(\Omega)\cap L^{\infty}(\Omega)\) for every \(k\) sufficiently large and satisfy
\[
\lim_{k\to +\infty} \big(\|u_k-u\|_{W^{1,p}(\Omega)} +\|H(u_k, \nabla u_k)-H(u, \nabla u)\|_{L^{1}(\Omega)} \big) =0.
\]
\end{proof}

Constant barriers can also be useful in the non-autonomous setting:

\begin{proposition}\label{lm-Giotto}
Assume that there exist a continuous function \(a:\R\to\R\) and a Carath\'eodory function \(g:\Omega\times\R^N\to\R\) such that
\[
L(x,t,\xi)=a(t)g(x,\xi)
\]
for a.e. \(x\in\Omega\) and every \((t,\xi)\in\R\times\R^N\). Assume also that
\(
g(\cdot,0)\in L^1(\Omega)\). Then,  \(\mathcal L\) is \(L^\infty(\Omega)\)-regular at every  \(u\in\mathcal A\) such that 
\(
(a\circ u) g(\cdot,0)\in L^1(\Omega)\).
\end{proposition}
\begin{proof}
For every \(k\geq1\), choose \(M_k^+\geq k\) and \(M_k^-\geq k\) such that
\[
|a(M_k^+)|
\leq\inf_{t\geq k}|a(t)|+\frac1k,
\qquad
|a(-M_k^-)|
\leq\inf_{t\leq-k}|a(t)|+\frac1k.
\]
Let \(u\in\mathcal A\) such that 
\((a\circ u) g(\cdot,0)\in L^1(\Omega)\).
If \(u(x)>M_k^+\), then \(u(x)\geq k\), and hence
\[
|a(M_k^+)|
\leq|a(u(x))|+\frac1k.
\]
Therefore
\[
\begin{aligned}
\int_{[u>M_k^+]}|L(x,M_k^+,0)|\,dx
&=\int_{[u>M_k^+]}|a(M_k^+)g(x,0)|\,dx\\
&\leq\int_{[u>M_k^+]}|a(u)g(x,0)|\,dx
+\frac1k\int_\Omega|g(x,0)|\,dx\\
&\longrightarrow0.
\end{aligned}
\]
Similarly, 
\[
\int_{[u<-M_k^-]}|L(x,-M_k^-,0)|\,dx
\longrightarrow0.
\]
Then, by Remark~\ref{rk-b=0},   the functions \(u_k:=\min(M_k^+, \max (u, -M_k^-))\) belong to \(W^{1,p}(\Omega)\cap L^{\infty}(\Omega)\) and give the desired approximating sequence.

\end{proof}

Constant barriers provide the most simple way  to  construct approximating sequences. However, such a construction is not always possible. One may then rely on Lemma~\ref{lm-classic} to produce barriers that are more elaborate. It is often convenient to define them as the sum of  constants and  functions with controlled oscillations. Let us reformulate Lemma~\ref{lm-classic} in this spirit:

\begin{lemma}\label{lm-translated-oscillators}
Let \(u\in\mathcal A\). Assume that there exist two sequences of numbers \((M_{k}^{\pm})_{k\geq 1}\) and two sequences of functions \((b_{k}^{\pm})_{k\geq 1}\subset W^{1,p}(\Omega)\cap L^{\infty}(\Omega)\) such that, denoting by \(\omega_k^\pm:=\operatorname*{ess\,osc}_{\Omega}b_k^\pm\), one has
\[
\lim_{k\to +\infty}M_k^\pm=0
\]
and
\[
\lim_{k\to +\infty}\int_{[u>M_k^+]}\max_{s\in [M_k^+, M_k^++\omega_k^+]} \overline L(x,s,\nabla b_k^+(x))\,dx = 0,
\]
\[
\lim_{k\to +\infty}\int_{[u<-M_k^-]}\max_{s\in [-M_k^--\omega_k^-,-M_{k}^{-}]} \overline L(x,s,\nabla b_k^-(x))\,dx=0.
\]
Then \(\mathcal L\) is \(L^\infty(\Omega)\)-regular at \(u\).
\end{lemma}
\begin{proof}
Define
\[
c_k^+:=
M_k^++b_k^+-\operatorname*{ess\,inf}_{\Omega}b_k^+
\]
and
\[
c_k^-:=
-M_k^--\bigl(\operatorname*{ess\,sup}_{\Omega}b_k^--b_k^-\bigr).
\]
Then
\[
c_k^\pm\in W^{1,p}(\Omega)\cap L^\infty(\Omega),
\]
\[
\operatorname*{ess\,inf}_{\Omega}c_k^+=M_k^+\longrightarrow+\infty,
\qquad
\operatorname*{ess\,sup}_{\Omega}c_k^-=-M_k^-\longrightarrow-\infty,
\]
and
\[
\nabla c_k^\pm=\nabla b_k^\pm
\quad\text{a.e. in }\Omega.
\]
Moreover,
\[
M_k^+\leq c_k^+\leq M_k^++\omega_k^+,
\qquad
-M_k^--\omega_k^-\leq c_k^-\leq-M_k^-
\quad\text{a.e. in }\Omega,
\]
and thus, 
\[
[u>c_k^+]\subset[u>M_k^+],
\qquad
[u<c_k^-]\subset[u<-M_k^-].
\]
Consequently,
\[
\int_{[u>c_k^+]}
\overline L(x,c_k^+,\nabla c_k^+)\,dx
\leq
\int_{[u>M_k^+]}\max_{s\in [M_k^+, M_k^++\omega_k^+]} \overline L(x,s,\nabla b_k^+(x))\,dx
\longrightarrow0
\]
and
\[
\int_{[u<c_k^-]}
\overline L(x,c_k^-,\nabla c_k^-)\,dx
\leq
\int_{[u<-M_k^-]}\max_{s\in [-M_k^--\omega_k^-,-M_{k}^{-}]} \overline L(x,s,\nabla b_k^-(x))\,dx
\longrightarrow0.
\]
The conclusion follows from Lemma~\ref{lm-classic}.
\end{proof}
\subsection{Barriers with small oscillation}

In the main result of this section Theorem~\ref{prop-della-francesca} below, we discard the Lavrentiev gap between \(W^{1,p}(\Omega)\) and  \(W^{1,p}(\Omega)\cap L^{\infty}(\Omega)\) under the assumption that for every \(t\in \R\),  the Lagrangian \(L(x,t,\xi)\) satisfies a good estimate involving only finitely many \(\xi\)'s. This result is obtained as a consequence of Lemma~\ref{lm-translated-oscillators} and the construction of suitable barriers. Such barriers must have small oscillations while  taking their gradients in a finite given set. This is indeed possible as  the following lemma shows:

\begin{lemma}\label{lm-oscillating-functions}
Let \(\xi_1, \dots, \xi_m\) be \(m\) vectors in \(\R^N\) such that \(0\) belongs to their convex hull.  Then for every  \(\varepsilon>0\), there exists \(b_{\varepsilon}\in W^{1,\infty}(\R^N)\) such that \(0\leq b_{\varepsilon}(x)\leq \varepsilon\) for every \(x\in \R^N\) and \(\nabla b_\varepsilon(x)\in \{\xi_1, \dots, \xi_m\}\) for a.e. \(x\in \R^N\).
\end{lemma}
\begin{proof} 
If one of the \(\xi_i\)'s is zero, then one can take \(b_{\varepsilon}\equiv 0\). Otherwise, by removing some of the \(\xi_i\)'s if necessary,  one can assume that \(0\) belongs to the relative interior of \(\operatorname{co}(\xi_1, \dots, \xi_m)\) and that the affine space \(E\) generated by \(\xi_1, \dots, \xi_m\)  has dimension \(m-1\). Assume for simplicity that \(E=\R^{m-1}\times \{0_{N-m+1}\}\) and write \(\xi_i=(\xi_{i}',0_{N-m+1})\) with \(\xi_{i}'\in \R^{m-1}\). We can thus apply \cite{Cellina} or \cite[Corollary 3.2]{DacorognaMarcellini} to construct a Lipschitz function \(b:[0,1]^{m-1}\to \R\) such that \(\nabla b(x')\in \{\xi_{1}', \dots, \xi_{m}'\}\) for a.e. \(x'\in (0,1)^{m-1}\) and \(b=0\) on the boundary of the cube \([0,1]^{m-1}\).
We next define \(\widetilde{b}=b+M\) where \(M:=\|b\|_{L^{\infty}([0,1]^{m-1})}\) so that \(0\leq \widetilde{b}\leq 2M \) on \([0,1]^{m-1}\).
We then extend \(\widetilde{b}\) on \(\R^{m-1}\) by \(1\) periodicity: for every \(x'\in \R^{m-1}\), one sets \(\overline{b}(x')=\widetilde{b}([x'])\), where \([x']\in [0,1]^{m-1}\) denotes the vector obtained by taking the fractional parts of the  coordinates of \(x'\). Finally, we extend \(\overline{b}\) to \(\R^{N}\) as follows:
\[
\forall (x',x'')\in \R^{m-1}\times \R^{N-m+1}, \qquad \widehat{b}(x',x'')=\overline{b}(x').
\]
Observe that for a.e. \((x',x'')\in \R^{m-1}\times \R^{N-m+1}\),
\[
\nabla \widehat{b}(x',x'')\in\{(\xi_1',0),\dots,(\xi_m',0)\}=\{\xi_1,\dots,\xi_m\}.
\]
Then the functions \(b_{\varepsilon}(x):=\frac{\varepsilon}{2M} \widehat{b}(2M x/\varepsilon)\)  satisfy all the required properties.
\end{proof}

The function \(b_\varepsilon\) defined in Lemma~\ref{lm-oscillating-functions} is the main technical ingredient in the construction of barriers arising in the next approximation result:

\begin{theorem}
\label{prop-della-francesca}
Let \(q\in L^1(\Omega)\) and \(h:\R\to[0,+\infty)\) a Borel map. Assume that for every \(s\in\R\), there exist an integer \(m_s\geq1\), vectors
\(
\xi_1^s,\dots,\xi_{m_s}^s\in\R^N
\)
such that
\[
0\in\operatorname{co}\{\xi_1^s,\dots,\xi_{m_s}^s\}
\]
and  \(\delta_s>0\) such that, for a.e. \(x\in\Omega\) and every \(t\in[s-\delta_s,s+\delta_s]\),
\begin{equation}\label{eq-local-balanced-control}
\sum_{j=1}^{m_s}\overline L(x,t,\xi_j^s)
\leq h(s)+q(x).
\end{equation}
Then \(\mathcal L\) is \(L^\infty(\Omega)\)-regular at every \(u\in\mathcal A\) such that \(h\circ u\in L^{1}(\Omega)\).
\end{theorem}
\begin{proof}
Let \(u\in \mathcal{A}\) such that \(h\circ u\in L^{1}(\Omega)\). By Lemma~\ref{lm-mkpm}, there exist two sequences \((M_{k}^+)_{k\geq 1}\) and \((M_{k}^-)_{k\geq 1}\) converging to \(+\infty\) and such that 
\begin{equation}\label{eq1041}
\lim_{k\to +\infty}|[u>M_k^+]|\,|h(M_k^+)|=0,\quad 
\lim_{k\to +\infty}|[u<-M_k^-]|\,|h(-M_k^-)|=0.
\end{equation}
By assumption, there exist an integer \(m_k^+\geq1\), vectors 
\(
\xi_{1,k}^+,\dots,\xi_{m_k^+,k}^+\in\R^N
\) and a number \(\delta_k^+>0\)
such that
\[
0\in\operatorname{co}
\{\xi_{1,k}^+,\dots,\xi_{m_k^+,k}^+\}
\]
and 
for a.e. \(x\in\Omega\) and every
\(t\in[M_k^+,M_k^++\delta_k^+]\),
one has
\[
\sum_{j=1}^{m_k^+}
\overline L(x,t,\xi_{j,k}^+)
\leq h(M_k^+)+q(x).
\]
By Lemma~\ref{lm-oscillating-functions}, there exists
\(
b_k^+\in W^{1,\infty}(\R^N)
\)
such that
\(
0\leq b_k^+\leq\delta_k^+
\)
and
\[
\nabla b_k^+(x)\in
\{\xi_{1,k}^+,\dots,\xi_{m_k^+,k}^+\}
\qquad\text{for a.e. }x\in\R^N.
\]
Set
\(
\omega_{k}^+:=\operatorname*{ess\,osc}_{\Omega}b_k^+\).
Since \(\omega_k^+\leq\delta_k^+\), for a.e. \(x\in\Omega\),
\[
\max_{t\in[M_k^+,M_k^++\omega_k^+]}
\overline L(x,t,\nabla b_k^+(x))
\leq h(M_k^+)+q(x).
\]
We thus obtain
\[
\int_{[u>M_k^+]}
\max_{t\in[M_k^+,M_k^++\omega_k^+]}
\overline L(x,t,\nabla b_k^+(x))\,dx\leq
|[u>M_k^+]|h(M_k^+)
+\int_{[u>M_k^+]}|q(x)|\,dx
.
\]
Using \eqref{eq1041} and the fact that  \(q\in L^1(\Omega)\), we deduce that the right-hand side converges to zero. Hence
\[
\lim_{k\to +\infty}\int_{[u>M_k^+]}
\max_{t\in[M_k^+,M_k^++\omega_k^+]}
\overline L(x,t,\nabla b_k^+(x))\,dx
=0.
\]
The lower barriers are constructed analogously. More precisely, applying the same argument to  \(M_k^-\), we get a finite family 
\(
\xi_{1,k}^-,\dots,\xi_{m_k^-,k}^-\in\R^N
\)
whose convex hull contains \(0\), and
\(
b_k^-\in W^{1,\infty}(\R^N)
\)
that satisfies
\[
\lim_{k\to +\infty}\int_{[u<-M_k^-]}
\max_{t\in[-M_k^--\omega_k^-,-M_k^-]}
\overline L(x,t,\nabla b_k^-(x))\,dx
=0,
\]
where \(\omega_k^-:=\operatorname*{ess\,osc}_{\Omega}b_k^-\).
The conclusion follows from Lemma~\ref{lm-translated-oscillators}.

\end{proof}

Theorem~\ref{prop-della-francesca} takes a simpler form when it is possible to choose the same family of vectors \(\xi_{1}^s, \dots, \xi_{m_s}^s\) for every \(s\in \R\): 
\begin{corollary}\label{prop-H(u,0)}
Assume that there exist
\(
q_0\in L^1(\Omega)
\), a Borel map \(h_0:\R\to [0,+\infty)\),
and vectors
\(
\xi_1,\dots,\xi_m\in\R^N
\)
such that
\[
0\in\operatorname{co}\{\xi_1,\dots,\xi_m\},
\]
and, for every \(j=1,\dots,m\), and every \(s\in\R\), for a.e. \(x\in \Omega\), 
\begin{equation}\label{eq1241}
|L(x,s,\xi_j)|\leq h_0(s)+q_0(x). 
\end{equation}
We further assume that for every \(j=1,\dots,m\), and every \(s\in\R\),
\begin{equation}\label{eq1245}
\lim_{t\to s} \|L(\cdot,t,\xi_j)-L(\cdot,s,\xi_j)\|_{L^{\infty}(\Omega)}=0.
\end{equation}
Then \(\mathcal L\) is \(L^\infty(\Omega)\)-regular at every \(u\in \mathcal{A}\) such that \(h_0\circ u\in L^{1}(\Omega)\).
\end{corollary}
\begin{proof}
By~\eqref{eq1245} and~\eqref{eq1241},
for every \(s\in \R\), there exists \(\delta_s>0\) such that for every \(t\in [s-\delta_s,s+\delta_s]\), for every \(j\in \{1, \dots,m\}\),  for a.e. \(x\in \Omega\),
\[
|L(x,t,\xi_j)| \leq |L(x,s,\xi_j)|+1\leq h_0(s)+q_0(x)+1.
\]
Hence, for every \(t\in [s-\delta_s, s+\delta_s]\), for a.e. \(x\in \Omega\),
\[
\sum_{j=1}^{m}
\overline{L}(x,t,\xi_j)\leq m(h_0(s)+q_0(x)+1)+\sum_{j=1}^{m}|\xi_j|.
\]
By the Fubini theorem and using also the fact that \(L\) is continuous with respect to the state variable \(t\), the above estimate also holds for a.e. \(x\in \Omega\), for every \(t\in [s-\delta_s, s+\delta_s]\).
The conclusion thus follows from Theorem~\ref{prop-della-francesca} taking \(m_s=m\), \(\xi_{j}^{s}=\xi_j\), \(\delta_s\), \(h=m h_0\) and \(q(x)= m (q_0(x)+1)+\sum_{j=1}^{m}|\xi_j|\).
\end{proof}

We recover Proposition~\ref{cor-lu0} by applying Corollary~\ref{prop-H(u,0)} with
\[
m=1,
\qquad
\xi_1=0,
\qquad
q_0:=0,
\]
and
\[
h_0(s):= |H(s,0)|.
\]

Even in the setting of autonomous Lagrangians \(L(x,t,\xi)=H(t,\xi)\), it may happen that \(H(u,0)\not\in L^{1}(\Omega)\) and still, that \(H(u,\xi_i)\in L^{1}(\Omega)\) for every \(\xi_i\) in a family which contains \(0\) in its convex hull. In such a situation, Proposition~\ref{cor-lu0} is not applicable and we need to rely on Corollary~\ref{prop-H(u,0)} to discard the Lavrentiev phenomenon, as  the forthcoming example shows. 

\begin{example}\label{ex-contrex-ct-barrier}
Let
\(\Omega:=B_1(0)\subset\R^2\) and choose \(\varphi\in \mathcal{C}_c^\infty(\R^2)\) such that
\(\varphi=2\) 
on \(\partial\Omega\). Let
\[
H(t,\xi):=t^4g(\xi),
\]
where \(g:\R^2\to[0,+\infty)\) is continuous, compactly supported in
\(B_1(0)\), and satisfies
\(g(0)=1\).
Consider
\[
u(x):=\frac{2}{\sqrt{|x|}}.
\]
Since
\[
\nabla u(x)=-\frac{x}{|x|^{5/2}},
\qquad
|\nabla u(x)|=\frac{1}{|x|^{3/2}},
\]
we have \(u\in W^{1,1}(\Omega)\).
Moreover, \(u\) has trace \(2\) on \(\partial\Omega\) and therefore
\(u\in W^{1,1}_\varphi(\Omega)\).
Since
\(|\nabla u(x)|>1\)  for every \(x\in\Omega\),
and \(g\) is supported in \(B_1(0)\), we obtain
\[
H(u,\nabla u)=0
\qquad\text{a.e. in }\Omega.
\]
Hence \(u\in\mathcal A\).

On the other hand,
\[
H(u,0)=u^4g(0)=\frac{16}{|x|^2}\notin L^1(\Omega).
\]
Nevertheless, for every \(\xi\notin B_1(0)\),
\[
H(u,\xi)=0.
\]
Choosing finitely many vectors outside \(B_1(0)\) whose convex hull contains
the origin, Corollary~\ref{prop-H(u,0)} shows that \(\mathcal L\) is
\(L^\infty(\Omega)\)-regular at \(u\).

Finally, for every \(t\geq2\),
\[
[u>t]
=
B_{4/t^2}(0),
\]
and hence
\[
|[u>t]|
=
\frac{16\pi}{t^4}.
\]
Since
\(H(t,0)=t^4\), 
we have
\[
|[u>t]|\,|H(t,0)|
=
\frac{16\pi}{t^4}t^4
=
16\pi.
\]
Therefore
\[
\lim_{t\to +\infty}|[u>t]|\,|H(t,0)|
\not=0.
\]
Thus the standard truncation method with constant barriers, described in
Remark~\ref{rk-b=0}, does not yield the \(L^\infty(\Omega)\)-regularity in
this example.
\end{example}

Corollary~\ref{prop-H(u,0)} has been formulated not just for autonomous Lagrangians, but also for \(x\)-dependent ones, a situation arising in the next example.

\begin{example}\label{ex-genuinely-vanishing-oscillation}
Let
\[
N=2,\qquad p=1,\qquad \Omega:=B_{e^{-1}}(0)\subset\R^2,
\]
and consider the zero boundary datum. For \(x\in\Omega\setminus\{0\}\), 
and define
\[
u(x):=\log\log\frac1{|x|}.
\]
Let
\[
e_1:=(1,0),\qquad K:=\{-e_1,e_1\}.
\]
After defining \(\nabla u(0):=0\), set
\[
d(x,\xi):=\operatorname{dist}\bigl(\xi,K\cup\{\nabla u(x)\}\bigr).
\]
Consider
\[
\ell(t):=\exp\bigl(2e^t-2t\bigr)
\]
and the Carath\'eodory integrand
\[
L(x,t,\xi):=\ell(t)+\ell(t)^2d(x,\xi)^2.
\]
We first observe that \(u\in W^{1,1}_0(\Omega)\). Indeed,
\[
\nabla u(x)
=
-\frac{x}{|x|^2\log(1/|x|)}
\qquad (x\neq0).
\]
Thus
\[
|\nabla u(x)|=\frac1{|x|\log(1/|x|)}
\qquad (x\neq0).
\]
Using polar coordinates and the change of variables \(y=\log(1/r)\), we get
\[
\int_\Omega|\nabla u|\,dx
=
2\pi\int_0^{e^{-1}}\frac1{\log(1/r)}\,dr
=
2\pi\int_1^{+\infty}\frac{e^{-y}}y\,dy
<+\infty.
\]
Moreover,
\[
\int_\Omega|u|\,dx
=
2\pi\int_0^{e^{-1}}\log\log\frac1r\,r\,dr
=
2\pi\int_1^{+\infty}\log y\,e^{-2y}\,dy
<+\infty.
\]
The trace of \(u\) on \(\partial\Omega\) is zero. Hence
\(u\in W^{1,1}_0(\Omega)\).
By construction,
\(d(x,\nabla u(x))=0\) for every  \(x\in\Omega\).
Therefore
\[
L(x,u(x),\nabla u(x))=\ell(u(x))=
\frac1{|x|^2\left(\log\frac1{|x|}\right)^2}.
\]
Consequently,
\[
\int_\Omega L(x,u(x),\nabla u(x))\,dx
=
2\pi\int_0^{e^{-1}}\frac1{r\left(\log\frac1r\right)^2}\,dr
=
2\pi\int_1^{+\infty}\frac1{y^2}\,dy
<+\infty.
\]
Thus, \(\ell\circ u \in L^{1}(\Omega)\) and  \(u\in\mathcal A\).
Since \(d(x,\pm e_1)=0\) for every \(x\in \Omega\), we also have for every \(s\in \R\), \(x\in \Omega\), 
\[
L(x,s, \pm e_1)=\ell(s),
\]
and thus, by continuity of \(\ell\), 
\[
\lim_{t\to s} \|L(\cdot, t, \pm e_1)-L(\cdot, s, \pm e_1)\|_{L^{\infty}(\Omega)}=0.
\]
Hence we can apply Corollary~\ref{prop-H(u,0)} with
\(q_0:=0\), \(h_0(t):=\ell(t)\) and with the two vectors \(\xi_1=e_1\), \(\xi_2=-e_1\)
to conclude that \(\mathcal L\) is \(L^\infty(\Omega)\)-regular at \(u\).

Finally, the criterion given in Remark~\ref{rk-b=0} cannot be applied. Indeed, since
\[
|\nabla u(x)|=\frac1{|x|\log(1/|x|)}\longrightarrow+\infty
\qquad\text{as }x\to0,
\]
we have
\(d(x,0)=1\)
for \(x\) sufficiently close to \(0\). Hence, near the origin,
\[
L(x,t,0)=\ell(t)+\ell(t)^2=\exp(2e^{t}-2t) + \exp(4e^t-4t).
\]
But
\(|[u\geq t]|=\pi \exp(-2e^t)\)
and therefore
\[
\lim_{t\to +\infty}\int_{[u\geq t]} L(x,t,0)\,dx = +\infty.
\]
Consequently,
the conclusion cannot be obtained from Remark~\ref{rk-b=0}.
\end{example}

The main assumption in Corollary~\ref{prop-H(u,0)}, namely the existence of some function \(h_0\) dominating \(L\) and such that \(h_0\circ u\) is summable, may fail to be true. In particular, this happens for autonomous Lagrangians \(H\) such that
\[
H(u,\xi)\notin L^1(\Omega)
\qquad\text{for every fixed }\xi\in\R^N.
\]
We illustrate this possibility in Example~\ref{ex:^q} below. In such a situation, we can still rely on the full strength of Theorem~\ref{prop-della-francesca}, which entitles one to choose  vectors \(\xi_{1}^s, \dots, \xi_{m_s}^s\) that depend on the state variable \(s\). In the autonomous setting, this leads to the following simple criterion:

\begin{corollary}\label{corollary-autonomous-symetric}
Assume that
\[
L(x,t,\xi)=H(t,\xi)
\]
for some continuous function \(H:\R\times\R^N\to\R\).
Assume  further that there exists \(C\geq 1\) such that for every \(t\in \R\), for every \(\xi\in \R^N\), one has
\[
|H(t,-\xi)|\leq C(|H(t,\xi)|+1).
\]
Then, \(\mathcal{L}\) is \(L^{\infty}(\Omega)\)-regular at every \(u\in \mathcal{A}\).
\end{corollary}
\begin{proof}
Given \(s\in \R\), let \(\xi_{1}^{s}\) be such that
\[
\overline{H}(s,\xi_{1}^s)\leq \inf_{\xi \in \R^N} \overline{H}(s,\xi)+1.
\]
We then define \(\xi_{2}^s=-\xi_{1}^s\). Then, \(0\in \operatorname{co}(\xi_{1}^{s},\xi_{2}^s)\) and
\[
h(s):=\overline{H}(s,\xi_{1}^{s})+ \overline{H}(s,\xi_{2}^{s})+1=\overline{H}(s,\xi_{1}^{s})+\overline{H}(s,-\xi_{1}^{s})+1\leq (1+C)(\overline{H}(s,\xi_{1}^s)+1),
\]
where the last inequality follows from the assumption. 

Hence
\[
h(s)
\leq (1+C) \left(\inf_{\xi \in \R^N} \overline{H}(s,\xi)+2\right).
\]
It follows that for every \(u\in \mathcal{A}\),  for a.e. \(x\in \Omega\),
\[
0\leq h(u(x)) \leq (1+C) \left(\overline{H}(u(x),\nabla u(x))+2\right).
\]
Consequently, \(h\circ u \in L^{1}(\Omega)\).

 Moreover, for every
\(s\in\R\), the function
\[
t\longmapsto
\overline H(t,\xi_1^s)+\overline H(t,-\xi_1^s)
\]
is continuous. Hence  there exists \(\delta_s>0\)
such that
\[
\overline H(t,\xi_1^s)+\overline H(t,-\xi_1^s)
\leq \overline H(s,\xi_1^s)+\overline H(s,-\xi_1^s)+1=h(s),
\]
whenever \(|t-s|\leq\delta_s\). Therefore all the assumptions of
Theorem~\ref{prop-della-francesca} (taking \(q=0\)) are satisfied.   
The conclusion
follows.
\end{proof}

We proceed to provide an example of an autonomous Lagrangian \(H\) together with some admissible \(u\in \mathcal{A}\) such that \(H(u, \xi)\not\in L^{1}(\Omega)\) for any \(\xi \in \R^N\). In particular, Corollary~\ref{prop-H(u,0)} does not apply. Still, this Lagrangian falls within the realm of Corollary~\ref{corollary-autonomous-symetric}.

\begin{example}\label{ex:^q}
Consider 
\[
H:(t,\xi)\mapsto \left||\xi|-\frac{1}{2}t^3\right|^q
\]
and \(u(x)=\frac{1}{|x|^{\frac{1}{2}}}\).

This example corresponds to 
\cite[Example~6.1]{MaricondaTreu2020}. There, the \(L^{\infty}(\Omega)\) approximation was proved by constructing
a specific radial approximating sequence, after observing that the standard
constant truncations do not preserve the energy in the relevant range of
exponents. 

More precisely, let
\[
\Omega:=B_1(0)\subset\R^2,
\qquad
1\leq p<\frac{4}{3},
\qquad
q\geq\frac{4}{3},
\]
and choose \(\varphi\in \mathcal{C}_c^\infty(\R^2)\) such that
\(\varphi=1\)
on \(\partial B_1(0)\).
Since
\[
u(x)=|x|^{-1/2},
\qquad
\nabla u(x)=-\frac{x}{2|x|^{5/2}},
\]
we have
\[
|\nabla u(x)|=\frac{1}{2|x|^{3/2}}
=\frac12u(x)^3.
\]
Moreover,
\[
\int_{B_1(0)}|\nabla u|^p\,dx
=
\frac{2\pi}{2^p}
\int_0^1r^{1-\frac{3p}{2}}\,dr
<+\infty
\]
because \(p<4/3\). Therefore \(u\in W^{1,p}_\varphi(\Omega)\).
Furthermore,
\[
H(u(x),\nabla u(x))
=
\left|
|\nabla u(x)|-\frac12u(x)^3
\right|^q
=0
\]
for a.e. \(x\in\Omega\). Hence
$u\in\mathcal A$.
On the other hand, for every fixed \(\xi\in\R^2\),
\[
H(u(x),\xi)
=
\left|
|\xi|-\frac{1}{2|x|^{3/2}}
\right|^q.
\]
For \(|x|\) sufficiently small, depending on \(\xi\), one has
\[
\frac{1}{2|x|^{3/2}}\geq2|\xi|,
\]
and therefore
\[
H(u(x),\xi)
\geq
\frac{1}{4^q|x|^{3q/2}}.
\]
Since \(q\geq4/3\),
\[
\int_{B_1(0)}H(u,\xi)\,dx=+\infty.
\]
Thus
\(H(u,\xi)\notin L^1(\Omega)\)  for every fixed \(\xi\in\R^2\) and Corollary~\ref{prop-H(u,0)} does not apply.

However, by Corollary~\ref{corollary-autonomous-symetric}, 
\(\mathcal L\) is \(L^\infty(\Omega)\)-regular at \(u\).

\end{example}

\subsection{Barriers with large oscillations}

The results presented in the preceding section rely on families of barriers with small oscillation. This section is devoted to a different criterion based on a single bounded function, whose oscillation need not be small, at the price of a local comparison assumption in the state variable. The corresponding construction is much more elementary but still yields very efficient barriers in some specific situations.

\begin{proposition}\label{prop-single-oscillator}
Assume that there exist \(b\in W^{1,p}(\Omega)\cap L^\infty(\Omega)\), \(q\in L^1(\Omega)\) and a Borel function \(h:\R\to[0,+\infty)\) such that
for a.e. \(x\in\Omega\) and every \(s,t\in\R\) satisfying
\[
|s-t|\leq \operatorname*{ess\,osc}_{\Omega}b,
\]
one has
\[
\overline L(x,s,\nabla b(x))
\leq q(x)+h(t).
\]
Then \(\mathcal L\) is \(L^\infty(\Omega)\)-regular at every \(u\in \mathcal{A}\) such that \(
h\circ u\in L^1(\Omega)\).
\end{proposition}
\begin{proof}
For every \(k\geq1\), choose \(M_k^\pm\geq k\) such that
\[
h(M_k^+)\leq\inf_{t\geq k}h(t)+1,
\qquad
h(-M_k^-)\leq\inf_{t\leq-k}h(t)+1.
\]
Set
\[
\omega:=\operatorname*{ess\,osc}_{\Omega}b,
\qquad
b_k^+:=b_k^-:=b.
\]
For a.e. \(x\in\Omega\) and every \(s\in[M_k^+,M_k^++\omega]\), our assumption implies
\[
\overline L(x,s,\nabla b_k^+(x))
\leq q(x)+h(M_k^+).
\]
Let \(u\in \mathcal{A}\) such that \(h\circ u\in L^{1}(\Omega)\).
For  \(x\in [u>M_k^+]\), the choice of \(M_k^+\) implies that the right-hand side is not lower than
\(q(x)+h(u(x))+1\). Hence
\[
\int_{[u>M_k^+]}\max_{s\in [M_k^+,M_k^++\omega]}\overline L(x,s,\nabla b_k^+(x))\,dx
\leq
\int_{[u>M_k^+]}(q(x)+h(u(x))+1)\,dx
\longrightarrow0.
\]

Similarly, 
\[
\int_{[u<-M_k^-]}\max_{s\in [-M_k^--\omega,-M_{k}^{-}]}\overline L(x,s,\nabla b_k^-(x))\,dx
\leq
\int_{[u<-M_k^-]}(q(x)+h(u(x))+1)\,dx
\longrightarrow0.
\]
The conclusion follows from Lemma~\ref{lm-translated-oscillators}.
\end{proof}

The most simple situation where Proposition~\ref{prop-single-oscillator} applies is given in the following corollary:

\begin{corollary}\label{cor-common-zero-oscillator}
Assume that there exists
\(
b\in W^{1,p}(\Omega)\cap L^\infty(\Omega)
\)
such that
\(
L(x,t,\nabla b(x))=0
\)
for a.e. \(x\in\Omega\) and every \(t\in\R\). Then \(\mathcal L\) is \(L^\infty(\Omega)\)-regular at every \(u\in\mathcal A\).
\end{corollary}
\begin{proof}
Apply Proposition~\ref{prop-single-oscillator} with
\[
q(x):=|\nabla b(x)|^p,
\qquad
h:=0.
\]
\end{proof}

We present a first consequence of Corollary~\ref{cor-common-zero-oscillator} in the case when the Lagrangian has a product form: 

\begin{corollary}\label{propositionH1H2}
Assume that there exist a continuous function \(a:\R\to\R\) and a Carath\'eodory function \(g:\Omega\times\R^N\to\R\) such that
\[
L(x,t,\xi)=a(t)g(x,\xi)
\]
for a.e. \(x\in\Omega\) and every \((t,\xi)\in\R\times\R^N\). Assume moreover that at least one of the following conditions holds:
\begin{enumerate}
\item[{\rm (H1)}] \(g(\cdot,0)\in L^\infty(\Omega)\) and there exists \(\alpha>0\) such that
\begin{equation}\label{eq647}
|g(x,\xi)|\geq\alpha
\qquad\text{for a.e. }x\in\Omega\text{ and every }\xi\in\R^N;
\end{equation}
\item[{\rm (H2)}] there exists \(\xi_0\in\R^N\) such that
\[
g(x,\xi_0)=0
\qquad\text{for a.e. }x\in\Omega.
\]
\end{enumerate}
Then \(\mathcal L\) is \(L^\infty(\Omega)\)-regular at every \(u\in\mathcal A\).
\end{corollary}
\begin{proof}
Let \(u\in\mathcal A\). Then
\[
a(u)g(\cdot,\nabla u)\in L^1(\Omega).
\]
Assume first that {\rm (H1)} holds. By \eqref{eq647},
\[
\alpha|a(u)|
\leq|a(u)g(\cdot,\nabla u)|
\quad\text{a.e. in }\Omega,
\]
and hence
\(a(u)\in L^1(\Omega)\).
Since \(g(\cdot,0)\in L^\infty(\Omega)\), it follows that
\[
L(\cdot,u,0)=a(u)g(\cdot,0)\in L^1(\Omega).
\]
The conclusion follows from Proposition~\ref{lm-Giotto}.
Assume now that {\rm (H2)} holds, and set
\[
b(x):=\langle\xi_0,x\rangle.
\]
Since \(\Omega\) is bounded, one has \(
b\in W^{1,p}(\Omega)\cap L^\infty(\Omega)\).
Moreover, for a.e. \(x\in\Omega\) and every \(t\in\R\),
\[
L(x,t,\nabla b(x))
=a(t)g(x,\xi_0)
=0.
\]
The conclusion follows from Corollary~\ref{cor-common-zero-oscillator}.
\end{proof}

In the very specific case of a product type autonomous Lagrangian, we can rely on the following criterion:

\begin{corollary}\label{prop-mantegna}
Consider an autonomous Lagrangian of the form \(H(t,\xi)=a(t)g(\xi)\), where 
 \(a:\R\to \R\) and \(g:\R^N\to \R \) are continuous. 
We  assume that
\begin{equation}\label{eq309}
\liminf_{|\xi|\to+\infty}|g(\xi)|>0.
\end{equation}
Then \(\mathcal{L}\) is \(L^{\infty}(\Omega)\)-regular at every \(u\in \mathcal{A}\).
\end{corollary}
\begin{proof}
Set
\[
\alpha:=\inf_{\xi\in\R^N}|g(\xi)|.
\]
If \(\alpha>0\), then
the conclusion follows from Corollary~\ref{propositionH1H2} under the assumption \((\textrm{H}1)\).
\noindent
If instead \(\alpha=0\), then by \eqref{eq309} and the continuity of \(g\), there exists \(\xi_0\in\R^N\) such that
\(g(\xi_0)=0\) and the conclusion thus follows from Corollary~\ref{propositionH1H2} under the assumption \((\textrm{H}2)\).
\end{proof}

\begin{remark}
The condition considered in
\cite[Proposition~3.7]{MaricondaTreu2020}, namely
\[
g(\xi)\geq c|\xi|
\qquad\text{for every }\xi\in\R^N,
\]
with \(c>0\), implies \eqref{eq309}. Thus, under the corresponding
product-type assumptions, Corollary~\ref{prop-mantegna} recovers that
result.
\end{remark}

\begin{remark}
Corollary~\ref{cor-common-zero-oscillator} also applies in the following   cases:
\begin{enumerate}
\item
\[
L(x,t,\xi)=R(a(x,t)g(\xi)),
\qquad
R(0)=0,
\qquad
g(\xi_0)=0;
\]
\item
\[
L(x,t,\xi)=K(a(x,t),\xi),
\qquad
K(s,\xi_0)=0
\quad \text{for every }s\in\R;
\]
\item
\[
L(x,t,\xi)=\sum_{i=1}^m a_i(x,t)g_i(\xi),
\qquad
g_i(\xi_0)=0
\quad\text{for every }i=1, \dots, m;
\]
\item \[
L(x,t,\xi)
=\lambda(t)|\xi-\nabla b(x)|^p, \quad \lambda\in \mathcal{C}^{0}(\R;[0,+\infty)), \quad b\in W^{1,p}(\Omega)\cap L^\infty(\Omega).
\]
\end{enumerate}
\end{remark}

Proposition~\ref{prop-single-oscillator} may be useful in some situations that are not covered by Corollary~\ref{cor-common-zero-oscillator}.
We proceed to present a simple criterion based on affine barriers, and that involve a function \(h:\R\to \R\) that satisfies a doubling condition.

\begin{corollary}\label{cor-single-direction}
Assume that there exist \(\xi_* \in \R^N\),  \(q_*\in L^{1}(\Omega)\) and a Borel function \(h_*:\R\to \R\) such that for a.e. \(x\in \Omega\) and for every \(t\in \R\),
\[
\overline{L}(x,t,\xi_*)\leq q_*(x)+h_*(t).
\]
Assume further that \(h_*\) satisfies a doubling condition: there exists \(C>0\) such that for every \(s, t\in \R\),
\[
|s|\leq 2|t|\Longrightarrow |h_*(s)|\leq C(|h_*(t)|+1). 
\]
Then, \(\mathcal L\) is \(L^\infty(\Omega)\)-regular at every \(u\in \mathcal{A}\) such that \(h_*\circ u\in L^{1}(\Omega)\).
\end{corollary}
\begin{proof}
Set
\[
b(x):=\langle\xi_*,x\rangle,
\qquad
\omega:=\operatorname*{ess\,osc}_{\Omega}b.
\]
For every \(s,t\in\R\) such that \(|s-t|\leq\omega\), one has 
\[
|s|\leq |t|+\omega \leq \max (2|t|, 2|\omega|).
\]
Since \(h_*\) satisfies the doubling condition, 
\[
|h_*(s)|\leq C(|h_*(t)|+1) + C(|h_*(\omega)|+1)\leq C_0(|h_*(t)|+1),
\]
with \(C_0=2C+ C|h_*(\omega)|\). 
Hence
\[
\overline L(x,s,\xi_*)
\leq
q_*(x)
+
C_0(h_*(t)+1).
\]
Therefore all the assumptions of Proposition~\ref{prop-single-oscillator}
are satisfied with \(q=q_*\)
and
\[
h(t):=C_0(h_*(t)+1)
.
\]
The conclusion follows.
\end{proof}

\section{Autonomous Lagrangians}
\label{section:autonomous}

In this section, we consider the specific case of autonomous integrands, namely those not depending on the spatial variable \(x\). Our main result, Theorem~\ref{cor-barycentric-subgradients} below, gives a sufficient condition for the non-occurence of the Lavrentiev gaps in \(L^{\infty}(\Omega)\),  in terms of the convex subdifferentials of the family of  functions \((H(t,\cdot))_{t\in \R}\). The heart of our strategy is contained in the following lemma stating the summability of a certain function involved in the subdifferential inequalities.

\begin{lemma}
\label{prop-affine-minorant}
Let \(H:\R\times\R^N\to[0,+\infty)\) be continuous. Assume that there exist a Borel locally bounded function
\[
K:\R\to[0,+\infty)
\]
such that, for every \(t\in\R\), there exists \(\zeta_t\in\R^N\) satisfying
\[
H(t,\xi)\geq K(t)+\langle\zeta_t,\xi\rangle
\qquad\text{for every }\xi\in\R^N.
\]
Then, for every \(u\in \mathcal{A}\), one has
\[
K\circ u\in L^1(\Omega).
\]
\end{lemma}
The proof of Lemma~\ref{prop-affine-minorant} essentially relies on the same 
 arguments as those used in the proofs of \cite[Proposition~5.1]{Bousquet-Annali} or \cite[Theorem~5.1]{MaricondaTreu2020}.
In the latter result, affine lower bounds in the gradient variable arise from
a non-oscillatory condition at infinity. In the former result, one requires that \(H\) be convex with respect to the gradient variable.  
\begin{proof}
For every \(t\in\R\), define
\[
G(t):=
\left\{
\zeta\in\R^N:
H(t,\xi)\geq K(t)+\langle\zeta,\xi\rangle
\text{ for every }\xi\in\R^N
\right\}.
\]
By assumption,  \(G(t)\) is nonempty.  Moreover, \(G(t)\) is closed, since it is an intersection of closed half-spaces. 
For every fixed \((t,\zeta)\in\R\times\R^N\), the map
\[
\xi\longmapsto H(t,\xi)-K(t)-\langle\zeta,\xi\rangle
\]
is continuous. Hence its nonnegativity on \(\mathbb Q^N\) is equivalent to its nonnegativity on \(\R^N\). This implies that
\[
G(t):=\bigcap_{\xi \in \mathbb{Q}^N}
\left\{
\zeta\in\R^N:
H(t,\xi)\geq K(t)+\langle\zeta,\xi\rangle
\right\}.
\]
We proceed to prove that \(G\) is Borel measurable; that is, for every compact set \(F\subset \R^N\), the set 
\[
G^{-1}(F)=\{t\in \R : G(t)\cap F\not=\emptyset\}
\]
is Borel measurable. Indeed, let \(F\subset \R^N\) be a compact set and let \((\zeta_j)_{j\geq 1}\) be a countable dense subset in \(F\). Then,
\[
G^{-1}(F)=\bigcap_{i\geq 1}\bigcup_{j\geq 1}\bigcap_{\xi\in \mathbb{Q}^N}\left\{t\in \R : H(t,\xi)\geq K(t)+\langle\zeta_j,\xi\rangle-\frac{1}{i}|\xi|\right\}.
\]
Since \(K\) is Borel measurable, this proves that \(G^{-1}(F)\) is a Borel set and thus \(G\) is Borel measurable. Together with the fact that \(G\) maps \(\R\) into the nonempty closed subsets of \(\R^N\), this implies that there exists a Borel measurable function \(\zeta:\R\to \R^N\) such that \(\zeta(t)\in G(t)\) for every \(t\in \R\), see e.g.~\cite[Theorem 6.5]{FonsecaLeoni}. Hence
\begin{equation}\label{eq2236b}
H(t,\xi)\geq K(t)+\langle\zeta(t),\xi\rangle
\qquad\text{for every }(t,\xi)\in\R\times\R^N.
\end{equation}

Testing \eqref{eq2236b} with \(\xi=e_i\) and \(\xi=-e_i\) respectively, gives for every \(t\in \R\),
\[
\zeta_i(t)\leq H(t,e_i)-K(t)\leq H(t,e_i)
\]
and
\[
-\zeta_i(t)\leq H(t,-e_i)-K(t)\leq H(t,-e_i).
\]
Therefore
\[
|\zeta_i(t)|
\leq
\max\{H(t,e_i),H(t,-e_i)\},
\qquad i=1,\dots,N.
\]
The right-hand side is locally bounded in \(t\), by the continuity of \(H\). 
We deduce that \(\zeta\) is locally bounded.

Let \(u\in \mathcal{A}\) and set
\[
F(x):=\langle\zeta(u(x)),\nabla u(x)\rangle.
\]
Evaluating the inequality \eqref{eq2236b} at \((t,\xi)=(u(x),\nabla u(x))\), we obtain
\begin{equation}\label{eq2071}
F\leq H(u,\nabla u)-K(u)\leq H(u,\nabla u)
\qquad\text{a.e. in }\Omega.
\end{equation}
Consequently,
\[
F_+\leq H(u,\nabla u)
\qquad\text{a.e. in }\Omega,
\]
and therefore
\[
F_+\in L^1(\Omega).
\]
Using the first inequality in \eqref{eq2071}, one gets
\begin{equation}\label{eq2093}
K(u)\leq H(u,\nabla u)-F
\qquad\text{a.e. in }\Omega.
\end{equation}
Fix
\[
T>\|\varphi\|_{L^\infty(\R^N)}.
\]
For \(k>T\), set
\[
v_k:=\min\{(u-T)_++T,k\}.
\]
Then
\[
v_k-T\in W^{1,p}_0(\Omega),
\qquad
\nabla v_k=\nabla u\,\chi_{\{T<u<k\}}
\quad\text{a.e. in }\Omega.
\]
Writing \(\zeta=(\zeta_1,\dots,\zeta_N)\), define
\[
Z_i(s):=\int_T^s\zeta_i(\tau)\,d\tau,
\qquad
Z:=(Z_1,\dots,Z_N).
\]
Since \(\zeta\) is locally bounded, \(Z\) is locally Lipschitz. Moreover, \(Z(T)=0\), \(v_k-T\in W^{1,p}_0(\Omega)\), and \(v_k\) is bounded. The Sobolev chain rule gives
\[
Z(v_k)\in W^{1,p}_0(\Omega;\R^N)
\]
and for every \(i\in \{1, \dots, N\}\),
\[
\partial_i (Z_i(v_k))(x)=\zeta_i(v_k(x))\partial_i v_k(x) \qquad \textrm{a.e. } x\in \Omega.
\]
Remember that \(\nabla v_k(x)=0\) for a.e. \(x\in \Omega\) such that \(v_k(x)\) is a non-differentiability point of \(Z_i\) and in that case, the right-hand side is defined as \(0\),  see~\cite{SerrinVarberg}.
 
Hence
\[
0
=
\int_\Omega\operatorname{div}Z(v_k)\,dx
=
\int_\Omega\langle\zeta(v_k),\nabla v_k\rangle\,dx.
\]
Since
\[
\langle\zeta(v_k),\nabla v_k\rangle
=
F\chi_{\{T<u<k\}}
\qquad\text{a.e. in }\Omega,
\]
we have
\[
\int_{\{T<u<k\}}F\,dx=0.
\]
Equivalently,
\[
\int_{\{T<u<k\}}F_+\,dx
=
\int_{\{T<u<k\}}F_-\,dx.
\]
Since \(F_+\in L^1(\Omega)\), the monotone convergence theorem yields
\(
F_-\chi_{\{u>T\}}\in L^1(\Omega)
\)
and
\[
\int_{\{u>T\}}F_-\,dx
=
\int_{\{u>T\}}F_+\,dx.
\]
Therefore
\(
F\chi_{\{u>T\}}\in L^1(\Omega)
\)
and
\[
\int_{\{u>T\}}F\,dx=0.
\]
Using the estimate \eqref{eq2093}, we obtain
\[
\int_{\{u>T\}}K(u)\,dx
\leq
\int_{\{u>T\}}
\bigl(H(u,\nabla u)-F\bigr)\,dx
=
\int_{\{u>T\}}
\bigl(H(u,\nabla u)\bigr)\,dx<+\infty.
\]
The negative tail is treated analogously. 
Finally, since \(K\) is locally bounded, it is bounded on \([-T,T]\). 
Combining the positive tail, the negative tail, and the bounded region, we conclude that
\[
K\circ u\in L^1(\Omega).
\]
\end{proof}

The above proposition applies in particular in the setting of subdifferential inequalities.

\begin{definition}
    Let \(h:\R^N\to \R\). Given \(\xi\in \R^N\), the convex subdifferential of \(h\) at \(\xi\), denoted by \(\partial h(\xi)\),   is the set of those \(\zeta\in \R^N\) such that
\[
\forall \xi'\in \R^N, \qquad h(\xi')\geq h(\xi)+\langle \zeta, \xi'-\xi\rangle.
\]
\end{definition}

In the remaining part of this section, given a function \((t,\xi)\in \R\times \R^N\mapsto H(t,\xi)\), we apply this definition to the functions \(\xi \mapsto H(t,\xi)\), for any fixed \(t\in \R\), and we denote by \(\partial H(t,\cdot)(\xi)\) the corresponding subdifferential at \(\xi\).

 Our first result on the non-occurence of the Lavrentiev gap   between \(W^{1,p}(\Omega)\) and \(W^{1,p}(\Omega)\cap L^{\infty}(\Omega)\) requires that  each function \(H(t,\cdot)\), with \(t\in \R\),  has a non-empty convex subdifferential at \(0\).

\begin{theorem}\label{th-real-valued-affine-minorant}
Assume that the continuous function \(H:\R\times\R^N\to\R\) satisfies the following two conditions:
\begin{itemize}
\item[(i)] for every \(t\in \R\), one has \(\partial H(t,\cdot)(0)\not=\emptyset\); that is,  there exists \(\zeta\in \R^N\) such that
\[
\forall \xi \in \R^N, \quad H(t,\xi)\geq H(t,0)+\langle \zeta, \xi \rangle.
\]
\item[(ii)] there exist \(c,d\in \R\) such that for every \(t\in \R\),
\[
H(t,0)\geq c|t|^{p^{\#}}+d.
\]
\end{itemize}
Then \(\mathcal{L}\)  is \(L^{\infty}(\Omega)\)-regular at every \(u\in \mathcal{A}\).
\end{theorem}
\begin{proof}
Set
\[
H_+:=\max\{H,0\},
\qquad
H_-:=\max\{-H,0\}.
\]
For every \(t\in\R\), let \(\zeta_t\in\R^N\) be as in assumption {\rm(i)}.
If \(H(t,0)>0\), then
\[
H_+(t,\xi)
\geq H(t,\xi)
\geq H(t,0)+\langle\zeta_t,\xi\rangle
=
H_+(t,0)+\langle\zeta_t,\xi\rangle
\]
for every \(\xi\in\R^N\). If instead \(H(t,0)\leq0\), then
\[
H_+(t,\xi)\geq0
=
H_+(t,0)+\langle0,\xi\rangle
\]
for every \(\xi\in\R^N\). Thus \(H_+\) satisfies the hypotheses of Lemma~\ref{prop-affine-minorant} with \(K(t)=H_+(t,0)\).

Let \(u\in \mathcal{A}\). Then
\(
H_+(u,\nabla u)\in L^1(\Omega)
\)
because
\[ 
0\leq H_+(u,\nabla u)
\leq
|H(u,\nabla u)|.
\]
Hence \(u\) belongs to the admissible class associated with \(H_+\).
Lemma~\ref{prop-affine-minorant}, applied to \(H_+\), yields
\[
H_+(u,0)\in L^1(\Omega).
\]
By Lemma~\ref{lm-Sobolev}, 
\(u\in L^{p^\#}(\Omega)\).
By assumption {\rm(ii)},
\[
H_-(t,0)
\leq H_+(t,0)+
|c||t|^{p^{\#}}+|d|
\qquad\text{for every }t\in\R.
\]
Consequently, since we have already proved that \(H_+(u,0)\in L^1(\Omega)\), we get
\[
H_-(u,0)\in L^1(\Omega).
\]
It follows that
\[
H(u,0)=H_+(u,0)-H_-(u,0)\in L^1(\Omega),
\]
and Proposition~\ref{cor-lu0} yields the conclusion.
\end{proof}
In particular, if \(H\) is nonnegative and convex with respect to the second variable, then \(\mathcal L\) is \(L^\infty(\Omega)\)-regular at every \(u\in \mathcal{A}\); we thus recover~\cite[Proposition~5.3]{Bousquet-Annali}.
Actually, the conclusion holds even if \(H\) is merely convex outside a ball, as a consequence of the following corollary:

\begin{theorem}
\label{cor-barycentric-subgradients}
Let \(H:\R\times\R^N\to[0,+\infty)\) be continuous. Let \(\xi_1,\dots,\xi_m\in\R^N\), and assume that there
exist \(\lambda_1,\dots,\lambda_m>0\) such that
\[
\sum_{i=1}^m\lambda_i=1,
\qquad
\sum_{i=1}^m\lambda_i\xi_i=0.
\] 
Assume that there exists a Borel function \(q:\R\to [0,+\infty)\) such that, for every
\(t\in\R\), there exist
\[
\zeta_{1,t},\dots,\zeta_{m,t}\in\R^N
\]
satisfying
\begin{equation}\label{eq2607}
H(t,\xi)\geq
H(t,\xi_i)+\langle\zeta_{i,t},\xi-\xi_i\rangle
\qquad\text{for every }\xi\in\R^N
\end{equation}
and 
\[
\sum_{i=1}^m
\lambda_i\langle\zeta_{i,t},\xi_i\rangle
\leq q(t).
\]
Then, for every \(u\in \mathcal{A}\) such that \(q\circ u \in L^{1}(\Omega)\), one has
\[
H(u,\xi_i)\in L^1(\Omega)
\qquad\text{for every }i=1,\dots,m.
\]
Consequently, \(\mathcal L\) is \(L^\infty(\Omega)\)-regular at \(u\).
\end{theorem}

\begin{proof}
Multiplying the supporting inequalities~\eqref{eq2607} by \(\lambda_i\) and summing, we obtain
\[
H(t,\xi)\geq
\sum_{i=1}^m\lambda_iH(t,\xi_i)
+
\left\langle
\sum_{i=1}^m\lambda_i\zeta_{i,t},\xi
\right\rangle
-q(t).
\]
Lemma~\ref{prop-affine-minorant}, applied with
\[
\widetilde{H}(t,\xi):=H(t,\xi)+q(t)
\]
and
\[
K(t):=\sum_{i=1}^m\lambda_iH(t,\xi_i),
\]
yields
\[
\sum_{i=1}^m\lambda_iH(u,\xi_i)\in L^1(\Omega).
\]
Since \(H\geq0\) and \(\lambda_i>0\), it follows that
\[
H(u,\xi_i)\in L^1(\Omega)
\qquad\text{for every }i=1,\dots,m.
\]
The conclusion is then a consequence of Corollary~\ref{prop-H(u,0)} applied with \(q_0=0\) and \(h_0(s)=\sum_{i=1}^{m}H(s,\xi_i)\).
\end{proof}

\begin{remark}\label{rem: exampleq}
Let \(H:\R\times \R^N\to [0,+\infty)\) be continuous. 
Assume that there exist \(0<r<R\) such that for every \(t\in \R\), the convex subdifferential  \(\partial H(t,\cdot)(\xi)\)  is nonempty at every \(\xi\) in the annulus \(B_R\setminus B_r\)  (here, \(B_R\)  denotes the open ball of radius \(R\) and center \(0\) in \(\R^N\), and similarly for \(B_r\)).  Assume further that there exist \(c>0\) and \(d\in \R\) such that
\begin{equation}\label{eq2657}
\forall t\in \R, \forall \xi\in B_R\setminus B_r, \qquad \operatorname{dist}(0,\partial H(t,\cdot)(\xi))\leq c|t|^{p^\#}+d.
\end{equation}
Then, the assumptions of Theorem~\ref{cor-barycentric-subgradients} are satisfied. Indeed, one only need to take  \(2\)  points \(\xi_1,  \xi_{2} \in B_{R}\setminus B_r\) such that \(0\in \operatorname{co}(\xi_1, \xi_{2})\). In view of~\eqref{eq2657}, for every \(i=1,2\) and every \(t\in \R\), there exists  \(\zeta_{i,t}\in \partial H(t,\xi_i)\) such that 
\[
|\zeta_{i,t}|\leq c|t|^{p^\#}+d.
\]
By the Cauchy-Schwarz inequality, it follows that
\[
\frac{1}{2}\langle \zeta_{1,t}, \xi_1\rangle +\frac{1}{2}\langle \zeta_{2,t}, \xi_2\rangle\leq R (c|t|^{p^\#}+d).
\]
Set \(q(t):=R(c|t|^{p^\#}+d)_+\).
By Lemma~\ref{lm-Sobolev}, for every \(u\in W^{1,p}(\Omega)\),  one has \(q\circ u \in L^{1}(\Omega)\).  This proves that all the assumptions of Theorem~\ref{cor-barycentric-subgradients} are satisfied. 
\end{remark}

\begin{conjecture}
In view of Corollary~\ref{corollary-autonomous-symetric} and the results of Section~\ref{section:autonomous}, we conjecture that in the autonomous case, \(\mathcal{L}\) is \(L^{\infty}(\Omega)\)-regular at every \(u\in \mathcal{A}\). 
\end{conjecture}

\section{Some classes of non-autonomous Lagrangians}
\subsection{State-independent Lagrangians}
The next result completely characterizes the Lagrangians not depending on the state variable for which the Lavrentiev gap in \(L^{\infty}(\Omega)\) does not occur.

\begin{theorem}\label{prop:nonaton2026}
Let \(g:\Omega\times\R^N\to\R\) be a Carath\'eodory function, let
\[
L(x,t,\xi):=g(x,\xi),
\]
and assume that \(\mathcal A\neq\varnothing\). Then the following properties are equivalent:
\begin{enumerate}
\item \(\mathcal L\) is \(L^\infty(\Omega)\)-regular at some \(u\in\mathcal A\);
\item there exists \(b\in W^{1,p}(\Omega)\cap L^\infty(\Omega)\) such that
\[
g(\cdot,\nabla b)\in L^1(\Omega);
\]
\item \(\mathcal L\) is \(L^\infty(\Omega)\)-regular at every \(u\in\mathcal A\).
\end{enumerate}
\end{theorem}
\begin{proof}
The implication \((3)\Rightarrow (1)\) is immediate. Assume that \((1)\) holds. By the definition of \(L^\infty(\Omega)\)-regularity, there exists a sequence
\[
(u_j)_{j\geq1}\subset\mathcal A\cap L^\infty(\Omega).
\]
Taking \(b:=u_1\), we obtain
\[
b\in W^{1,p}(\Omega)\cap L^\infty(\Omega),
\qquad
g(\cdot,\nabla b)\in L^1(\Omega),
\]
and hence \((1)\Rightarrow(2)\).
Assume now that \((2)\) holds, and let \(u\in\mathcal A\). Set
\[
q(x):=|g(x,\nabla b(x))|+|\nabla b(x)|^p,
\qquad
h:=0.
\]
Then \(q\in L^1(\Omega)\), \(h\circ u\in L^1(\Omega)\), and, for a.e. \(x\in\Omega\) and every \(s,t\in\R\),
\[
\overline L(x,s,\nabla b(x))
=|g(x,\nabla b(x))|+|\nabla b(x)|^p
=q(x).
\]
Thus all the assumptions of Proposition~\ref{prop-single-oscillator} are satisfied. Hence \(\mathcal L\) is \(L^\infty(\Omega)\)-regular at \(u\). Since \(u\in\mathcal A\) was arbitrary, \((2)\Rightarrow (3)\).
\end{proof}

\subsection{Comparison with autonomous Lagrangians}
Some of the results obtained in the autonomous setting extend to non-autonomous Lagrangians by comparison. The following proposition is a variant of Corollary~\ref{th-real-valued-affine-minorant}.

\begin{proposition}\label{prop:comparison2026}
We assume that there exists a continuous function \(H:\R\times \R^N \to [0,+\infty)\), a Carath\'eodory function \(g:\Omega\times \R^N \to [0,+\infty)\) and a constant \(A>0\) such that for a.e. \(x\in \Omega\), for every \((t,\xi)\in \R\times \R^N\):
\begin{equation}\label{eq575}
H(t,\xi)\leq L(x,t,\xi) \leq A H(t,\xi) +g(x,\xi).
\end{equation}
We further assume that \(g(\cdot, 0) \in L^{1}(\Omega)\) and that for every \(t\in \R\), there exists \(\zeta \in \R^N\) such that
\[
\forall \xi \in \R^N, \qquad H(t,\xi)\geq H(t,0) + \langle \zeta, \xi \rangle.
\]
Then \(\mathcal{L}\) is \(L^{\infty}(\Omega)\)-regular at every \(u\in \mathcal{A}\) such that \(g(x,\nabla u)\in L^{1}(\Omega)\).
\end{proposition}
\begin{proof}
Let \(u\in\mathcal A\) be such that
\[
g(x,\nabla u)\in L^1(\Omega).
\]
By the first inequality in \eqref{eq575}, one has
\[
H(u,\nabla u)\in L^1(\Omega).
\]
The affine-minorant assumption and
Lemma~\ref{prop-affine-minorant} applied to \(K(t)=H(t,0)\) yield
\[
H(u,0)\in L^1(\Omega).
\]
In view of Proposition~\ref{cor-lu0} and its proof, there exist two sequences \((M^{+}_k)_{k\geq 1}\) and \((M^{-}_k)_{k\geq 1}\) which tend to \(+\infty\) and such that the functions \(u_k:=\min(M^{+}_k, \max (u, -M^{-}_k))\) belong to \(W^{1,p}(\Omega)\cap L^{\infty}(\Omega)\) and satisfy
\[
\lim_{k\to +\infty} \big(\|u_k-u\|_{W^{1,p}(\Omega)} +\|H(u_k, \nabla u_k)-H(u, \nabla u)\|_{L^{1}(\Omega)} \big) =0.
\]

Next, observe that
\[
g(x,\nabla u_k)=
\begin{cases}
g(x,\nabla u),&-M^{-}_k<u<M^{+}_k,\\
g(x,0),&u\geq M^{+}_k\text{ or }u\leq-M^{-}_k,
\end{cases}
\quad\text{a.e. in }\Omega.
\]
Since \(M^{\pm}_k\to+\infty\), one has
\[
\lim_{k\to +\infty}\left|[u\geq M^{+}_k]\cup[u\leq-M^{-}_k]\right|=0.
\]
Moreover,
\[
|g(x,\nabla u_k)-g(x,\nabla u)|
\leq
\bigl(g(x,0)+g(x,\nabla u)\bigr)
\chi_{[u\geq M^{+}_k]\cup[u\leq-M^{-}_k]}.
\]
Since \(g(\cdot,0),g(\cdot,\nabla u)\in L^1(\Omega)\), the dominated convergence theorem yields
\[
g(\cdot,\nabla u_k)\longrightarrow g(\cdot,\nabla u)
\quad\text{in }L^1(\Omega).
\]
Furthermore, for a.e. \(x\in\Omega\), one has
\[
u_k(x)=u(x),
\qquad
\nabla u_k(x)=\nabla u(x)
\]
for all sufficiently large \(k\), and thus,
\[
\lim_{k\to +\infty}L(x,u_k,\nabla u_k)
=
L(x,u,\nabla u)
\quad\text{for a.e. }x\in\Omega.
\]
Moreover, by \eqref{eq575},
\[
0\leq L(x,u_k,\nabla u_k)
\leq
A H(u_k,\nabla u_k)+g(x,\nabla u_k).
\]
The families
\[
\bigl(H(u_k,\nabla u_k)\bigr)_k
\qquad\text{and}\qquad
\bigl(g(x,\nabla u_k)\bigr)_k
\]
are uniformly integrable, since they converge in \(L^1(\Omega)\). Hence
\[
\bigl(L(x,u_k,\nabla u_k)\bigr)_k
\]
is uniformly integrable. The Vitali convergence theorem therefore gives
\[
L(x,u_k,\nabla u_k)
\longrightarrow
L(x,u,\nabla u)
\quad\text{in }L^1(\Omega).
\]

\end{proof}

\subsection{Truncations by composition}
In this last section, we present an alternate approach to discard the Lavrentiev gap between \(W^{1,p}(\Omega)\) and \(W^{1,p}(\Omega)\cap L^\infty(\Omega)\). We construct barriers of the form \(\pi\circ u\), where \(\pi:\R\to \R\) is a suitable bounded function. Let us first recall the following standard result:

\begin{lemma}\label{lm-other-truncation}
Let \((\pi_i)_{i\geq 1}\) be a family of Lipschitz functions from \(\R\) into \(\R\) such that
\begin{enumerate}
\item There exists two sequences \((m_{i}^+)_{i\geq 1},(m_{i}^-)_{i\geq 1}\subset \R^+\) converging to  \(+\infty\)  and such that \(\pi_i(t)=t\) for every \(t\in [-m_{i}^-, m_{i}^+]\);
\item There exists \(M>0\) such that \(\|\pi_{i}'\|_{L^{\infty}(\R)}\leq M\) for every \(i\geq 1\).
\end{enumerate}
Then for every \(u\in W^{1,p}(\Omega)\), the sequence \((\pi_i(u))_{i\geq 1}\) converges to \(u\) in \(W^{1,p}(\Omega)\).
\end{lemma}
\begin{proof}
First, each \(u_i:=\pi_i(u)\) belongs to \(W^{1,p}(\Omega)\) and satisfies 
\[
\nabla u_i(x) = \pi_{i}'(u(x))\nabla u(x) \qquad \textrm{ a.e. } x\in \Omega
\] 
where the right-hand side is defined as \(0\) for a.e. \(x\in \Omega\) such that \(u(x)\)  is a non-differentiability point of \(\pi_i\).
Then
\[
\int_{\Omega}(|u_i-u|^p+|\nabla u_i - \nabla u|^p)\,dx = \int_{[|u|>m_i]}(|u_i-u|^p + |\pi_{i}'(u)\nabla u-\nabla u|^p)\,dx.
\]
On the set \([u>m_i]\), we write
\[
|\pi_i(u(x))-u(x)|=\left|\int_{m_i}^{u(x)}(\pi_{i}'(s)-1)\,ds\right| \leq (u(x)-m_i)(M+1)\leq (M+1)u(x)
\]
and
\[
|\pi_{i}'(u(x))\nabla u(x)-\nabla u(x)|\leq (M+1)|\nabla u(x)|.
\]
Using  similar estimates on the set \([u<-m_i]\), one gets
\[
\int_{\Omega}(|u_i-u|^p+|\nabla u_i - \nabla u|^p)\,dx \leq (M+1)^p\int_{[|u|>m_i]}(|u|^p+|\nabla u|^p)\,dx.
\]
By the dominated convergence theorem, one can conclude that
\[
\lim_{i\to +\infty}\int_{\Omega}(|u_i-u|^p+|\nabla u_i - \nabla u|^p)\,dx=0.
\]
\end{proof}

The preceding lemma gives the Sobolev convergence of compositions that
coincide with the identity on expanding intervals. We apply it to the following sawtooth functions:

\begin{definition}\label{def-sawtooth}
Let \(m^+, m^->0\), \(\varepsilon>0\) and \(\theta^+, \theta^-\in [0,1]\). We define
\(\pi:\R\to\R\) by
\[
\pi(t):=
\begin{cases}
-m^-+\theta^-\operatorname{dist}(t+m^-,2\varepsilon\mathbb Z),
& t<-m^-,\\
t,
& -m^- \leq t\leq m^+,\\
m^+-\theta^+\operatorname{dist}(t-m^+,2\varepsilon\mathbb Z),
& t>m^+.
\end{cases}
\]
Then \(\pi\) is bounded and Lipschitz, \(\pi(t)=t\) for every \(t\in [-m^-,m^+]\) and
\[
\pi'(t)\in\{-\theta^+,\theta^+\}
\qquad\text{a.e. } t\geq m^+,
\]
\[
\pi'(t)\in\{-\theta^-,\theta^-\}
\qquad\text{a.e. } t\leq -m^-,
\]
and
\[
\pi([m^+,+\infty))
\subset[m^+-\varepsilon,m^+],
\]
\[
\pi((-\infty,-m^-])
\subset[-m^-,-m^-+\varepsilon].
\]
\end{definition}
The function \(\pi\) coincides with the
identity on \([-m^-,m^+]\) and oscillates with slopes \(\pm\theta^\pm\) on the set \((-\infty, -m^-]\cup [m^+,+\infty) \).
\begin{figure}[H]
\centering
\includegraphics[width=0.85\textwidth]{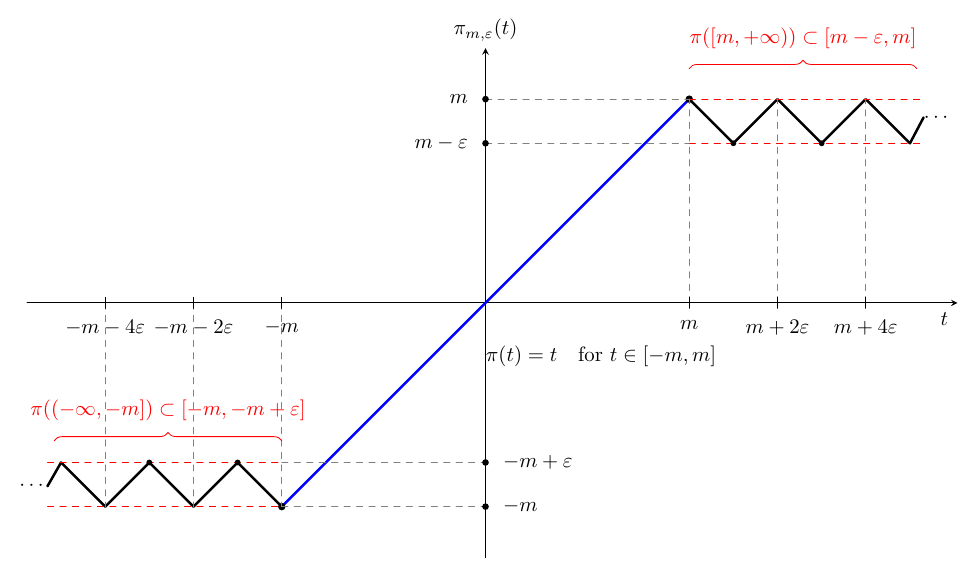}
\caption{The sawtooth truncation
\(\pi\). It coincides with the identity
on \([-m^-,m^+]\), oscillates between \(m^+-\varepsilon\) and \(m^+\) on \([m^+,+\infty)\), and between \(-m^-\) and \(-m^-+\varepsilon\) on 
\((-\infty,m^-]\). For simplicity, the figure depicts the symmetric case
\(\theta^+=\theta^-=1\), \(m^+=m^-=m\).}
\label{fig-pi-sawtooth}
\end{figure}
We are now in a position to state and prove the main result of this section, where we obtain a new sufficient condition to discard the Lavrentiev gaps in \(L^{\infty}(\Omega)\), in terms of the growth of the Lagrangians with respect to the state variable.

\begin{theorem}\label{coro-other-truncations}
Assume that there exist \(\theta^+, \theta^- \in [0,1]\), two  sequences \((m_{i}^+)_{i\geq1}, (m_{i}^-)_{i\geq1}\subset \R^+\) converging to
\(+\infty\),   a sequence  \((\varepsilon_i)_{i\geq1}\) of positive numbers, a summable function \(q\in L^{1}(\Omega)\)
and a constant \(C\geq 0\) such that, for every \(i\geq1\), for a.e.
\(x\in\Omega\), and for every \(\xi\in\mathbb R^N\),
\begin{equation}\label{eq3098}\max_{s\in[m_{i}^+-\varepsilon_i,m_{i}^+]}\bigl(|L(x,s,\theta^+ \xi)|+ |L(x,s,-\theta^+\xi)|\bigr)
\leq
C\inf_{s\in[m_{i}^+,+\infty)}|L(x,s,\xi)|+q(x)
\end{equation}
and
\begin{equation}\label{eq3107}
\max_{s\in[-m_{i}^-,-m_{i}^-+\varepsilon_i]}\bigl(|L(x,s,\theta^-\xi)|+
|L(x,s,-\theta^-\xi)| \bigr)
\leq
C\inf_{s\in(-\infty,-m_{i}^-]}|L(x,s,\xi)|+q(x).
\end{equation}
Then \(\mathcal L\) is \(L^\infty(\Omega)\)-regular at every
\(u\in\mathcal A\).
\end{theorem}
\begin{proof}
 Let \(u\in\mathcal A\).
 For every \(i\geq 1\), consider the bounded Lipschitz function \(\pi_i\) 
given by Definition~\ref{def-sawtooth} with the parameters \(m^\pm:=m_{i}^\pm, \theta^{\pm}, \varepsilon:=\varepsilon_i\) and set
\[
u_i:=\pi_i\circ u.
\]
Since \(\pi_i\) is bounded,
\[
u_i\in L^\infty(\Omega).
\]
By Lemma~\ref{lm-other-truncation}, the sequence \((\pi_i(u))_{i\geq 1}\) converges to \(u\) in \(W^{1,p}(\Omega)\). 
Furthermore, for \(i\) sufficiently large,
\[
m_{i}^{\pm}>\|\varphi\|_{L^\infty(\mathbb R^N)},
\]
and therefore
\[
u_i\in W^{1,p}_\varphi(\Omega).
\]
We proceed to prove that \((L(x,u_i, \nabla u_i))_{i\geq 1}\) converges to \(L(x,u,\nabla u)\) in \(L^{1}(\Omega)\).
By the Sobolev chain rule, for a.e. \(x\in \Omega\), one has \(\nabla u_i(x)=\pi_{i}'(u(x))\nabla u(x)\) (where as usual the right-hand side is defined to be \(0\) for  a.e. \(x\in \Omega\) such that \(u(x)\) is a non-differentiability point of \(\pi_i\)).
On the set \([|u|<m_i]\), one has in fact \(u_i=u\) and \(\nabla u_i=\nabla u\). Therefore 
\[
|L(x,u_i(x), \nabla u_i(x))|=|L(x,u(x),\nabla u(x))|.
\]
The properties of  \(\pi_i\) imply that for a.e. \(x\in [u\geq m_{i}^+]\),   one has \(u_i(x)\in [m_{i}^+-\varepsilon,m_{i}^+]\) and \(\nabla u_i(x)=\pm\theta^+ \nabla u(x)\) and thus by \eqref{eq3098}, one gets
\[
|L(x,u_i(x), \nabla u_i(x))|\leq C\bigl(|L(x,u(x),\nabla u(x))|+q(x)\bigr).
\]
A similar estimate holds on the set \([u\leq -m_{i}^-]\). Hence,  for a.e. \(x\in \Omega\), for every \(i\geq 1\),
\[
|L(x,u_i(x),\nabla u_i(x))|
\leq
C_0\bigl(
 |L(x,u(x),\nabla u(x))|+q(x)
\bigr),
\]
where \(C_0:=\max\{1,C\}\).
In particular,
\(L(x,u_i,\nabla u_i)\in L^1(\Omega)\) so that 
\(u_i\in\mathcal A\). Moreover, 
\[
\int_\Omega
|L(x,u_i,\nabla u_i)-L(x,u,\nabla u)|\,dx
\leq (1+C_0)
\int_{[u\not\in (m_{i}^-, m_{i}^+)]}
(|L(x,u,\nabla u)|+q(x))\,dx
\longrightarrow 0.
\]
The proof is complete.
\end{proof}

We first apply Theorem~\ref{coro-other-truncations} in the case when one can rely on the monotonicity of the Lagrangians with respect to the state variable.

\begin{corollary}\label{cor-monotone-tails}
Let
\[
L:\Omega\times\R\times\R^N\to[0,+\infty)
\]
be a Carath\'eodory integrand. Assume that there exist \(C>0\) and \(t_0>0\) such that
\[
L(x,t,-\xi)\leq C L(x,t,\xi)
\]
for a.e. \(x\in\Omega\) and every \((t,\xi)\in\R\times\R^N\), and such that, for a.e. \(x\in\Omega\) and every \(\xi\in\R^N\), the function
\[
t\longmapsto L(x,t,\xi)
\]
is nondecreasing on \((t_0,+\infty)\) and nonincreasing on \((-\infty,-t_0)\). Then \(\mathcal L\) is \(L^\infty(\Omega)\)-regular at every \(u\in\mathcal A\).
\end{corollary}
\begin{proof}
Choose \(m_k\to+\infty\) and \(\varepsilon_k\in(0,1)\) such that
\(
m_k-\varepsilon_k>t_0\).
Let \(C_0:=\max\{1,C\}\).
If \(s\in[m_k-\varepsilon_k,m_k]\), then monotonicity gives
\[
L(x,s,\xi)
\leq L(x,m_k,\xi)
=\inf_{r\geq m_k}L(x,r,\xi),
\]
while
\[
L(x,s,-\xi)
\leq L(x,m_k,-\xi)
\leq C L(x,m_k,\xi).
\]
Hence
\[
\max\left\{
\max_{s\in[m_k-\varepsilon_k,m_k]}L(x,s,\xi),
\max_{s\in[m_k-\varepsilon_k,m_k]}L(x,s,-\xi)
\right\}
\leq
C_0\inf_{r\geq m_k}L(x,r,\xi).
\]
Similarly, 
\[
\max\left\{
\max_{s\in[-m_k,-m_k+\varepsilon_k]}L(x,s,\xi),
\max_{s\in[-m_k,-m_k+\varepsilon_k]}L(x,s,-\xi)
\right\}
\leq
C_0\inf_{r\leq-m_k}L(x,r,\xi).
\]
Therefore the assumptions of
Theorem~\ref{coro-other-truncations} are satisfied with
\[
m_k^+=m_k^-:=m_k, \qquad \theta^+=\theta^-=1, \qquad q=0,
\]
and the conclusion follows.
\end{proof}

If in Corollary~\ref{propositionH1H2}, the function \(g(x,\cdot)\) is even and satisfies \(\inf_{\Omega\times \R^N}|g|=0\) but this infimum is not attained, then neither {\rm (H1)} nor {\rm (H2)} in Corollary~\ref{propositionH1H2} are satisfied. However, even in this situation, it is often possible to conclude, as the following consequence of Theorem~\ref{coro-other-truncations} shows.

\begin{corollary}
\label{cor-dellafrancesca}
We assume that there exists a continuous function \(a:\R\to \R\) and a Carathéodory function \(g:\Omega\times \R^N\to \R \) such that \(L(x,t,\xi)=a(t)g(x,\xi)\) for every \((x,t,\xi)\in \Omega \times \R\times \R^N\). 
Assume that \(g(\cdot,0)\in L^1(\Omega)\) and that there exists \(c>0\) such that
\begin{equation}\label{eq427}
|g(x,-\xi)|\leq c|g(x,\xi)|
\qquad\text{for a.e. }x\in\Omega\text{ and every }\xi\in\mathbb R^N.
\end{equation}
Then \(\mathcal{L}\) is \(L^{\infty}(\Omega)\)-regular at every \(u\in \mathcal{A}\).
\end{corollary}
\begin{proof}
Set
\[
\alpha^+:=\liminf_{t\to+\infty}|a(t)|,
\qquad
\alpha^-:=\liminf_{t\to-\infty}|a(t)|,
\]
and define
\[
\theta^+:=\begin{cases}
1 & \textrm{ if } \alpha_+>0,\\
0 & \textrm{ if } \alpha_+=0,
\end{cases}
\qquad
\theta^-:=\begin{cases}
1 & \textrm{ if } \alpha_->0,\\
0 & \textrm{ if } \alpha_-=0.
\end{cases}
\]
Assume first that \(\alpha_+>0\). 
There exists \((m_{i}^{+})_{i\geq 1}\) such that for every \(i\) sufficiently large \(m_{i}^{+}\geq i\) and 
\[
0<|a(m_{i}^{+})|\leq 
2 \inf_{t\geq i}|a(t)|. 
\]
By continuity of \(a\),  there exists \(\varepsilon_i>0\) such that
\[
\max_{s\in [m_{i}^+-\varepsilon_i, m_{i}^+]}|a(s)| \leq 2 |a(m_{i}^+)|.
\]
Using~\eqref{eq427}, this implies that for a.e. \(x\in \Omega\) and for every \(\xi\in \R^N\), 
\begin{align}
\max_{s\in[m_{i}^+-\varepsilon_i,m_{i}^+]}\bigl(|a(s)g(x,\theta^+ \xi)|+ |a(s)g(x,-\theta^+\xi)|\bigr)
&\leq (1+c)\big(\max_{s\in[m_{i}^+-\varepsilon_i,m_{i}^+]}|a(s)|\bigr)|g(x,\xi)|\\
&\leq
4(1+c)\inf_{t\geq i}|a(t)| |g(x,\xi)|=4(1+c) \inf_{t\geq i}|L(x,t,\xi)|\\
&\leq 4(1+c) \inf_{t\geq m_{i}^+}|L(x,t,\xi)|.
\end{align}

When \(\alpha^+=0\), there exists \((m_{i}^+)_{i\geq 1}\) such that for every \(i\geq 1\), one has \(m_{i}^+\geq i\) and \(|a(m_{i}^+)|\leq 1\). By continuity of \(a\),  there exists \(\varepsilon_i>0\) such that
\[
\max_{s\in [m_{i}^+-\varepsilon_i, m_{i}^+]}|a(s)| \leq  |a(m_{i}^+)|+1\leq 2.
\]
Then, since \(\theta^+=0\) in that case, 
\[
\max_{s\in[m_{i}^+-\varepsilon_i,m_{i}^+]}\bigl(|L(x,s,\theta^+ \xi)|+ |L(x,s,-\theta^+\xi)|\bigr)
= 2\max_{s\in[m_{i}^+-\varepsilon_i,m_{i}^+]}|a(s)||g(x,0)| \leq 4|g(x,0)|. 
\]
In both cases, we thus have
\[
\max_{s\in[m_{i}^+-\varepsilon_i,m_{i}^+]}\bigl(|L(x,s,\theta^+ \xi)|+ |L(x,s,-\theta^+\xi)|\bigr) \leq 4(1+c)\inf_{s\in [m_{i}^+,+\infty)}|L(x,s,\xi)|+4|g(x,0)|.
\]
This proves that \eqref{eq3098} is satisfied, by taking \(C=4(1+c)\)  and \(q(x)=4|g(x,0)|\).

Similarly, by considering the two cases \(\alpha^->0\) or \(\alpha^-=0\), one can construct a sequence \((m_{i}^-)_{\geq 1}\) which tends to \(+\infty\) and such that, up to decreasing \(\varepsilon_i\),
\[
\max_{s\in[-m_{i}^-, -m_{i}^- +\varepsilon_i]}\bigl(|L(x,s,\theta^- \xi)|+ |L(x,s,-\theta^-\xi)|\bigr)\leq C\inf_{s\in (-\infty, -m_{i}^-]}|L(x,s,\xi)|+4|g(x,0)|.
\]
The conclusion then follows from Theorem~\ref{coro-other-truncations}.
\end{proof}

\appendix
\section{Technical lemmata}

\begin{lemma}\label{lm-Sobolev}
Let \(\Omega\) be an open set in \(\R^N\), \(p\in [1,N]\) and let \(\varphi\in W^{1,p}(\R^N)\). Then, for every \(u\in W^{1,p}_\varphi(\Omega)\), one has \(
u\in L^{p^\#}(\Omega)\).
\end{lemma}
\begin{proof}
Since \(u-\varphi\in W^{1,p}_0(\Omega)\), its extension by zero belongs to
\(W^{1,p}(\R^N)\). Since also \(\varphi\in W^{1,p}(\R^N)\), the function
\(u\) admits an extension belonging to \(W^{1,p}(\R^N)\). 
Therefore
\[
u\in L^{p^\#}(\Omega):
\]
if \(p<N\), this follows from the Sobolev inequality in \(\R^N\), whereas
if \(p=N\), it follows directly from \(u\in W^{1,N}(\Omega)\). 
\end{proof}

\begin{lemma}\label{lm-mkpm}
Let \(h:\R\to \R\) be a Borel map and let \(u:\Omega\to \R\) be a measurable map. If \(h\circ u \in L^{1}(\Omega)\), then there exist two sequences \((M_{k}^+)_{k\geq 1}\) and \((M_{k}^-)_{k\geq 1}\) converging to \(+\infty\) and such that
\[
\lim_{k\to +\infty}|[u\geq M_{k}^{+}]| |h(M_{k}^{+})|=0, \quad \lim_{k\to +\infty}|[u\leq -M_{k}^{-}]| |h(M_{k}^{-})|=0.
\]
\end{lemma}
\begin{proof}
For every \(k\geq 1\), choose \(M_{k}^+\geq k\) and \(M_{k}^-\geq k\) in a such a way that
\[
|h(M_{k}^+)| \leq \inf_{t\geq k}|h(t)|+1, \quad  |h(-M_{k}^-)| \leq \inf_{t\leq -k}|h(t)|+1.
\]
Then,
\[
|[u\geq M_{k}^{+}]| |h(M_{k}^{+})|\leq \int_{[u\geq k]}(|h(u(x))|+1)\,dx.
\]
Since \(h\circ u\in L^{1}(\Omega)\), the dominated convergence theorem implies that the right hand side converges to \(0\). Hence 
\[
\lim_{k\to +\infty}|[u\geq M_{k}^{+}]| |h(M_{k}^{+})|=0.
\]
Similarly,
\[
\lim_{k\to +\infty}|[u\leq -M_{k}^{-}]| |h(-M_{k}^{-})|=0.
\]
The proof  is complete. 
\end{proof}

\section*{Declaration on the use of generative artificial intelligence}
During the preparation of the revised version of this manuscript, the authors used OpenAI's ChatGPT as an editorial and critical-review tool. In particular, it was used to suggest linguistic and stylistic improvements, reorganize parts of the exposition, check the consistency of notation and cross-references, assist with \LaTeX\ formatting, and identify passages in which an argument or a justification required further clarification. Furthermore, it  suggested possible reformulations or extensions of some statements and proofs. It also provided Example~\ref{ex-genuinely-vanishing-oscillation}. Every mathematical argument, statement, correction, and reference included in the manuscript was subsequently examined, verified, and approved by the authors. The authors take full responsibility for the content of the manuscript.
\bibliographystyle{amsplain}
\bibliography{references}
\end{document}